\documentclass{amsart}
\usepackage[utf8]{inputenc}
\usepackage{amsmath,amssymb,amsthm, graphics, comment,bm}
\usepackage{mathtools}
\usepackage{amsthm}
\usepackage{diagmac2}
\usepackage{textcomp}
\usepackage{color}
\usepackage[alphabetic]{amsrefs}
\usepackage[all,cmtip,color,matrix,arrow]{xy}
\usepackage{yfonts}
\usepackage{tikz}
\usepackage{stmaryrd}
\usepackage{mathrsfs}
\usepackage{enumerate}
\usepackage{hyperref}
\usepackage[mathscr]{euscript}

\calclayout

\newcommand{\dotr}{\mbox{$\boldsymbol{\cdot}$}}

\newcommand{\Diff}{\operatorname{Diff}}

\newcommand{\bG}{\mathbb{G}}
\newcommand{\bL}{\mathbb{L}}

\newcommand{\bQ}{\mathbb{Q}}

\newcommand{\sC}{\mathscr{C}}
\newcommand{\sD}{\mathscr{D}}

\newcommand{\sG}{\mathscr{G}}
\newcommand{\sH}{\mathscr{H}}

\newcommand{\sK}{\mathscr{K}}
\newcommand{\sL}{\mathscr{L}}

\newcommand{\sN}{\mathscr{N}}
\newcommand{\sO}{\mathscr{O}}

\newcommand{\sS}{\mathscr{S}}

\newcommand{\sX}{\mathscr{X}}
\newcommand{\sZ}{\mathscr{Z}}

\newcommand{\cC}{\mathcal{C}}
\newcommand{\cD}{\mathcal{D}}
\newcommand{\cE}{\mathcal{E}}
\newcommand{\cF}{\mathcal{F}}
\newcommand{\cG}{\mathcal{G}}
\newcommand{\cH}{\mathcal{H}}
\newcommand{\cI}{\mathcal{I}}
\newcommand{\cK}{\mathcal{K}}
\newcommand{\cL}{\mathcal{L}}
\newcommand{\cM}{\mathcal{M}}
\newcommand{\cN}{\mathcal{N}}
\newcommand{\cO}{\mathcal{O}}
\newcommand{\cP}{\mathcal{P}}
\newcommand{\cQ}{\mathcal{Q}}
\newcommand{\cR}{\mathcal{R}}
\newcommand{\cS}{\mathcal{S}}
\newcommand{\cU}{\mathcal{U}}

\newcommand{\cX}{\mathcal{X}}
\newcommand{\cY}{\mathcal{Y}}
\newcommand{\cZ}{\mathcal{Z}}
\newcommand{\Ass}{\operatorname{Ass}}
\newcommand{\E}{\operatorname{Ext}}
\newcommand{\spec}{\operatorname{Spec}}

\newcommand{\supp}{\operatorname{Supp}}

\newcommand{\Hom}{\operatorname{Hom}}

\newcommand{\Div}{\operatorname{Div}}
\newcommand{\Pic}{\operatorname{Pic}}

\newcommand{\Aut}{\operatorname{Aut}}

\newcommand{\Id}{\operatorname{Id}}

\newcommand{\bmu}{\bm{\mu}}
\newtheorem{theorem}{Theorem}[section]
\newtheorem{Teo}[theorem]{Theorem}

\newtheorem{Lemma}[theorem]{Lemma}
\newtheorem{Oss}[theorem]{Observation}
\newtheorem*{Oss'}{Observation}
\newtheorem{Cor}[theorem]{Corollary}
\newtheorem{Prop}[theorem]{Proposition}

\theoremstyle{definition}
\newtheorem{EG}[theorem]{Example}
\newtheorem{Def}[theorem]{Definition}
\newtheorem{Notation}[theorem]{Notation}
\newtheorem{Remark}[theorem]{Remark}

\begin{document}

\title[Twisted stable pairs]{Stable pairs with a twist and gluing morphisms for moduli of surfaces}
\author{Dori Bejleri}
\email{bejleri@math.harvard.edu}
\author{Giovanni Inchiostro}
\email{giovanni@math.brown.edu}

\begin{abstract}We propose an alternative definition for families of stable pairs $(X,D)$ over a possibly non-reduced base when $D$ is reduced, by replacing $(X,D)$ with an appropriate orbifold pair $(\cX,\cD)$. This definition of a stable family ends up being equivalent to previous ones, but has the advantage of being more amenable to the tools of deformation theory. Moreover, adjunction for $(\cX,\cD)$ holds on the nose; there is no correction term coming from the different. This leads to the existence of functorial gluing morphisms for families of stable surfaces and functorial morphisms from $(n + 1)$ dimensional stable pairs to $n$ dimensional polarized orbispace. As an application, we study the deformation theory of some surface pairs. 
\end{abstract}

\maketitle

\section{Introduction}

Since the introduction of the space of stable curves by Deligne and Mumford \cite{DM69}, the theory of modular compactifications of moduli spaces of varieties has held a central role in algebraic geometry. The class of stable pairs $(X,D)$, first introduced in dimension $2$ by Koll\'ar and Shepherd-Barron \cite{KSB}, give a natural generalization of stable curves to higher dimensions. Building on significant advances in the minimal model program, boundedness, properness, and projectivity of the coarse moduli space has been proven for the space of stable pairs \cites{KSB, kol, al1, karu, BCHM, HX, HMX, KP, kol1}. 

One subtle aspect is the notion of a flat family of stable pairs over an arbitrary base. The first difficulty is that the natural polarization $K_X + D$ on a stable pair is only a $\mathbb{Q}$-divisor, so the associated sheaf $\cO_X(K_X + D)$ is not locally free. The second difficulty is that the family of divisors $D$ is not flat when the coefficients are small. Koll\'ar introduced moduli problems to address these difficulties in \cite{Kol08} and \cite{Kol_div19} respectively. However, infinitesimal deformations and obstruction theories for these moduli problems are not well understood in general since the moduli functor does \emph{not} simply parametrize flat families of pairs $(X,D) \to B$ which are fiberwise stable (see e.g. \cite[Chapter 6.1]{kol1} and \cite[Problem 13]{Kol_div19} for discussions).

In this paper we propose an alternative solution to the first difficulty for pairs $(X,D)$ with $D$ reduced, which is more amenable to the tools of deformation and obstruction theory.\footnote{In this setting the family of divisors is automatically flat (see Corollary \ref{Cor:D:is:a:FLAT:family}).} Building upon the work of Abramovich and Hassett \cite{AH}, we replace $(X,D)$ with an associated \emph{twisted stable pair} $(\cX, \cD)$. Admissible families of stable pairs become just flat families of twisted stable pairs with no extra condition.

\subsection{Advantages of twisted stable pairs}
\begin{enumerate} 

\item We can compute admissible deformations and obstructions of stable pairs via usual deformations and obstructions of twisted stable pairs (Corollary \ref{cor:intro:def} and Section \ref{sec:intro:def}).
\item We extend the gluing formalism of \cite[Theorem 5.13]{Kollarsingmmp} to families over arbitrary bases giving \emph{functorial} morphisms that describe the boundary of the moduli of stable surfaces (Theorem \ref{thm:intro3}).
\item Adjunction for twisted stable pairs is well behaved in families and induces a \emph{functorial} morphism from the moduli of twisted stable pairs to the moduli of canonically polarized orbifolds in one dimension lower (Theorem \ref{thm:intro2}). 
\end{enumerate}

\subsection{Moduli of twisted stable pairs}\label{sec:intro:tsp}

One way to address the issue that $\cO_X(K_X + D)$ is \emph{not} a line bundle is to choose an appropriate reflexive power which is a line bundle; but different choices would result in different moduli problems. Therefore Koll\'ar introduced the following, more canonical, condition on a flat family $\pi : (X,D) \to B$: the sheaves $\omega_\pi^{[m]}(mD)$ are flat and commute with base change for \emph{every} $m$ (see \cite{Kol08}). This is \emph{not} a fiberwise condition, but rather a global condition on the family.

In \cite{AH}, Abramovich and Hassett showed that imposing Koll\'ars condition is equivalent to working with an associated orbifold $\cX \to B$. When $D=0$, they constructed a proper Deligne-Mumford stack parametrizing such canonically polarized twisted stable varieties. Our first result is to extend this approach, incorporating a non-zero reduced divisor $D$, and simplifying the construction of the moduli space of stable pairs in this setting:

\begin{theorem}[Theorems \ref{Teo:we:have:alg:stack} and \ref{Teo:the:KSBA:stack:is:proper:DM}]\label{thm:intro1} There exists a proper Deligne-Mumford stack $\cK_{n,v}$ parametrizing $n$-dimensional twisted stable pairs $(\cX, \cD)$ of volume $v$ where $\cX$ is an orbifold with no stabilizers in codimension one, $\cD$ is a reduced divisor, $K_\cX + \cD$ is a Cartier ample divisor, and $(\cX, \cD)$ has semi-log canonical singularities.
\end{theorem} 

To incorporate the divisor into the framework of \cite{AH}, we replace $D$ with the morphism of sheaves $\cO_X(K_X) \to \cO_X(K_X + D)$. In particular, we consider pairs $(\cX, \phi : \omega_\cX \to \cL)$ where $(\cX, \cL)$ is a polarized orbispace (see Definition \ref{Def:polorbitspace}) and $\phi$ is a nonzero morphism which cuts out the divisor $\cD$ as the support of its cokernel (see \ref{sec:obj:H}, \ref{Def:moduli:functor} and \ref{Def:KSBA:component}).

The upshot, by taking the coarse moduli space, is that twisted stable pairs $(\cX \to B, \phi : \omega_\cX \to \cL)$ are equivalent to flat families $\pi : (X,D) \to B$ with stable fibers and that satisfy Koll\'ar's condition for the log canonical sheaves: the sheaves $\omega_\pi^{[m]}(mD)$ are flat and compatible with base change for every $m$. 

\begin{Cor}\label{cor:intro:def}
Let $(\cX, \cD)$ be a twisted stable pair with trivial stabilizers in codimension $1$ and coarse space $(X,D)$. Then infinitesimal deformations of the associated pair $(\cX, \phi : \omega_\cX \to \cL)$ are in bijection with infinitesimal deformations of the stable pair $(X,D)$ that satisfy Koll\'ar's condition. 
\end{Cor} 

\subsection{Applications to gluing and adjunction}
The formalism of twisted stable pairs simplifies the adjunction formula. There is no longer a correction term coming from the singularities of $X$:

\begin{theorem}[Proposition \ref{Prop:no:different:on:stack} and Corollary \ref{Cor:different:only:from:double:locus:arbitrary:dim}]\label{thm:intro2}
Let $(\cX, \cD)$ be a twisted stable pair. Then $\cL\big|_{\cD} \cong \omega_\cD$. Moreover, let $(\cX \to B, \phi : \omega_\cX \to \cL)$
be a twisted stable pair over an arbitrary base and suppose that $\cD \to B$ is an $S_2$ morphism. Then $(\cD \to B,\cL\big|_{\cD})$ is a family of canonically polarized orbifolds with semi-log canonical singularities.
\end{theorem} 

The first part of Theorem \ref{thm:intro2} is essentially saying that the different (see Section \ref{Subsection:different}) of $(X,D)$ is replaced by the orbifold structure of $(\cX, \cD)$. We study the precise relation between the different on $(X,D)$ and the stabilizers on $(\cX, \cD)$ at the end of Section \ref{Section:local:and:global:structure:of:TSP}.
For surfaces, the second part of Theorem \ref{thm:intro2} gives us a morphism from $\cM_{2,v}$ to the moduli space of orbifold stable curves. 

As an application, we show that in the case of surfaces, the gluing results of \cite{Kol13}*{Theorem 5.13} can be extended to the orbifold setting and can be done functorially for families over an arbitrary base. More precisely, we construct the algebraic stack $\cG_{2,v}$ of gluing data consisting of triples $(\cX, \cD, \tau : \cD^n \to \cD^n)$ where $(\cX, \cD)$ is an object of $\cK_{2,v}$ as in Theorem \ref{thm:intro1}, $\cD^n$ is the normalization of $\cD$, and $\tau : \cD^n \to \cD^n$ is a generically fixed point free involution which preserves the preimage of the nodes of $\cD$. Denote by $\cK_{2,v}^\omega$ the moduli stack of stable surfaces. 

\begin{theorem}[Theorems \ref{Teo:gluing:morphisms} and \ref{Teo:gluing:morphism:between:moduli}]\label{thm:intro3} There is a morphism $\cG_{2,v} \to \cK_{2,v}^\omega$ which on the level of closed points sends a triple of gluing data as above to the stable surface given by gluing the coarse space $D \subset X$ along the involution $\tau$ as in \cite{Kol13}*{Theorem 5.13}.
\end{theorem}

Theorem \ref{thm:intro3} is an analogue for stable surfaces of the gluing morphisms that describe the boundary of the moduli space of stable curves $\overline{\cM}_g$ \cite{Knu}. In Section \ref{subsec:boundary:strata} we show that there exists a finite stratification of $\cK_{2,v}^\omega$ such that each boundary stratum is the image of a family of gluing data under the morphism in Theorem \ref{thm:intro3}. 

\subsection{Applications to deformation problems}\label{sec:intro:def}
As a first application of our approach, we compute the admissible deformations and obstructions for a plt surface pair $(X,D)$. If we denote with $(\cX,\cD)$ its associated twisted stable pair, then we have equivalencies $$\operatorname{Def}(\cX, \omega_{\cX} \to \cL)\cong \operatorname{Def}(\cX, \omega_{\cX}\otimes \cL^{-1} \to \cO_{\cX})\cong \operatorname{Def}(\cX, \cG \to \cO_{\cX})$$ where $\cG = \omega_{\cX}\otimes \cL^{-1}$. But by Lemma \ref{Lemma:relation:singularities:different} when $(X,D)$ is a plt surface, the surface $\cX$ is Gorenstein. Therefore, $\cG$ is a line bundle, so $\cG \to \cO_{\cX}$ induces a morphism $\psi:\cX\to[\mathbb{A}^1/\bG_m]$. In particular, $$\operatorname{Def}(\cX, \omega_{\cX} \to \cL) = \operatorname{Def}(\psi).$$ The explicit description of the stabilizers of $\cX$ around $\cD$ shows that $\psi$ is representable, so by \cite{Ols06} the space $\operatorname{Def}(\psi)$ latter is controlled by the cotangent complex $\mathbb{L}_{\cX/[\mathbb{A}^1/\bG_m]}$.

In Section \ref{section_example} we use this approach to explicitly compute the admissible deformations and obstructions of a K3 surface with an $A_1$ singularity at a point $p$, and a divisor $D$ which is \emph{not} Cartier passing through the singular point. Using twisted stable pairs we prove that the Koll\'ar's moduli space of stable pairs is smooth of dimension 19 at such a point and agrees with the moduli space of the quasi-polarized K3 pair obtained by resolving the singularity. In forthcoming work we will extend this approach to study the deformations and obstructions for pairs in higher dimensions. 

\subsection{Outline} In Section \ref{Section:background} we recall the background we need on singularities of the MMP \cites{Kol13, KM98} and polarized orbispaces \cite{AH}. In Section \ref{Section:moduli:TSP} we give the definition of twisted stable pair (Definition \ref{Def:moduli:functor}) and we prove Theorem \ref{thm:intro1}.
In Section \ref{Section:local:and:global:structure:of:TSP} we study the local properties of polarized orbispaces in order to compare our moduli functor with the $\mathbb{Q}$-Gorenstein deformations used by \cite{Hac04} and to prove Theorem \ref{thm:intro2}. In Section \ref{Section:gluing} we prove Theorem \ref{thm:intro3} on gluing morphisms. In \ref{section_example} we give an example deformation theoretic computation using twisted stable pairs. 

\subsection*{Acknowledgements} The authors would like to thank Dan Abramovich, Shamil Asgarli, Brendan Hassett, J\'anos Koll\'ar, S\'andor Kov\'acs, Davesh Maulik, Jonathan Wise, and Chenyang Xu for many helpful conversations. We thank David Rydh for informing us of Theorem \ref{Teo:Rydh} and its proof. The first author is supported by an NSF Postdoctoral Fellowship DMS-1803124 and was partially supported by NSF grant DMS-1759514. The second author is partially supported by NSF grant DMS-1759514. The second author also thanks the UC Berkeley Department of Mathematics for their hospitality while part of this research was conducted.  

\subsection*{Conventions} We work over an algebraically closed field $k$ of characteristic 0.
Unless otherwise specified, all the stacks will be of finite type over $k$. A stack $\cX$ has property $\cP$ generically if there is an open embedding $U \to \cX$ which intersects all the irreducible components of $\cX$, such that the points of $U$ have property $\cP$. 
When algebraic stack $\cX$ admits a coarse moduli space, unless otherwise stated we will denote its coarse moduli space by $X$. When we say that a diagram of stacks commutes, we mean it 2-commutes.

\section{Background on the minimal model program and polarized orbispaces}\label{Section:background}
This section is divided into three subsections. In the first one, we begin by recalling the properties we need about the singularities of the MMP.
Then we introduce the analogous singularities for a Deligne-Mumford (DM) stack.
In the second subsection we recall the definition of the different.
Finally, in the last subsection we recall the relevant constructions and definitions from \cite{AH}.
\subsection{Singularities of the MMP} 
In this subsection we recall the properties of the singularities that appear in the MMP, which are relevant for the rest of the paper.
For a more detailed exposition, and for the definitions of lc, slc, plt, demi-normal and Du Bois, we refer to \cite{Kollarsingmmp} and \cite{KM98}.

We begin by introducing the following notation:
\begin{Notation}\label{Notation:canonical:as:pushforward}
We say that an open subset $U \subset X$ is \emph{big} if the complement of $U$ has codimension at least $2$ in $X$. Furthermore,
given $f:X\to  B$ a flat morphism of DM stacks, we denote with $U(f)$ the $f$-Gorenstein locus.
\end{Notation}

We highlight the following useful observation: 

\begin{Oss}\label{Oss:Gor:locus:pulls:back:to:G:locus} Let $f:\mathcal{X} \to B$ be a flat family of DM stacks and let $B' \to B$ a morphism. Consider the pullback $ f':\mathcal{X}':=\mathcal{X}\times_B B' \to B'$ and let $h:\mathcal{X}' \to \mathcal{X}$ be the first projection. Then $h^{-1}(U(f))=U(f')$.
\end{Oss}

The following lemma is known to the experts. For convenience we include a proof. 
\begin{Lemma}\label{Lemma:unique:arrow:reflexive:hull}
Consider a flat $S_2$ morphism $\pi: \cX \to B$ from a DM stack $\cX$ to a scheme, and let $U \subseteq \cX$ be an open substack which is big on each fiber.
Let $\cE$ be a reflexive sheaf on $\cX$, and let $\cF$ be a coherent sheaf on $\cX$.
Then the restriction map
 $\operatorname{Hom}_{\cX}(\cF, \cE) \to \operatorname{Hom}_U(\cF_{|U}, \cE_{|U})$ is an isomorphism.
\end{Lemma}
\begin{proof} Up to replacing $\cX$ with an atlas, we can assume it is a scheme, which we denote by $X$. Let $j:U \to X$ be the inclusion of $U$.
Since $\cE$ is reflexive and $U$ is big along each fiber, from \cite{HK}*{Proposition 3.6.1} the morphism
$\cE \to j_*(j^*\cE)$ is an isomorphism. Therefore, by the adjunction between $j_*$ and $j^*$, we have
$
 \Hom_X(\cF, \cE) \cong \Hom_X(\cF, j_*(j^*\cE)) \cong \Hom_U(j^*(\cF), j^*(\cE))
$.
 \end{proof}
Consider $f:X\to B$ a flat separated morphism of locally Noetherian DM stacks with $S_2$ and pure $d$-dimensional fibers. Let $\omega_f^{\dotr}$ be the relative dualizing complex. 

\begin{Def}
We define the \emph{relative canonical sheaf} $\omega_{X/B}$ or $\omega_f$ to be the sheaf $ \cH^{-d}(\omega_f^{\dotr})$. 
\end{Def}

When we assume that the fibers are Gorenstein in codimension $1$, then $\omega_{X/B}$ agrees with the pushforward $\iota_*\omega_{U(f)/B}$ where $\iota:U(f) \hookrightarrow X$ the inclusion of the relative Gorenstein locus (see \cite{YN18}*{Section 5}). In this case $\omega_{X/B}$ is in fact a reflexive sheaf \cite{YN18}*{Proposition 5.6}. 

Next we generalize the definitions for singularities of pairs to DM stacks:
 \begin{Def}\label{Def:slc:stack}
 Consider a pair $(\cX,\sum a_i \cD_i)$ consisting of a DM stack $\cX$ and reduced equi-dimensional closed substacks $\cD_i$ of codimension 1 with $a_i \in \mathbb{Q}_{(0,1]}$. We say that the pair $(\cX,\sum a_i \cD_i)$ is \emph{log canonical} or \emph{lc} (resp. \emph{semi-log canonical} or \emph{slc}) if there is an \'etale cover $f : Y \to \cX$ by a scheme such that $(Y, \sum a_i f^*\cD_i)$ is log canonical (resp. semi-log canonical).
 \end{Def}

 The following observation follows from \cite{Kollarsingmmp}*{2.40, 2.41 and Corollary 2.43} (see also \cite{Kollarsingmmp}*{Proposition 2.15}). 
 
\begin{Oss}\label{Oss:definition:lc:doesnotdepend:on:covering}
Consider a pair $(X,\sum a_i D_i)$ consisting of a demi-normal scheme $X$ and pure codimension $1$ reduced subschemes $D_i$.
Let $f : Y \to X$ be an étale surjective morphism, or a finite surjective morphism that is étale in codimension 1, with $Y$ demi-normal. Then the pair $(X,\sum a_i D_i)$ is lc (resp. slc) if and only if $(Y,\sum a_i f^*D_i)$ is lc (resp. slc).
\end{Oss}
\noindent The main consequences of Observation \ref{Oss:definition:lc:doesnotdepend:on:covering} are the following:
\begin{itemize}
 \item Definition \ref{Def:slc:stack} does not depend on the choice of the \'etale cover $Y$;
 \item Consider a pair $(\cX,\sum a_i\cD_i)$ with $\cX$ a demi-normal DM stack that is a scheme in codimension 2.
 Let $X$ (resp. $D_i$) be the coarse space of $\cX$ (resp. $\cD_i$). Then
 $(X,\sum a_i D_i)$ is lc (resp. slc) if and only if $(\cX,\sum a_i \cD_i)$ is lc (resp. slc). 
\end{itemize}

\subsection{The different}
In this subsection we recall the definition of different. We refer the reader to \cite{Kollarsingmmp}*{Definition 2.34 and Chapter 4} for further details.

Suppose $(X,D+\Delta)$ is an lc pair, where $D$ is a divisor with coefficient 1 and $\Delta$ is a $\mathbb{Q}$-divisor. If $X$ and $D$ are \emph{smooth},
then the usual adjunction formula gives a canonical isomorphism \begin{equation}\label{eqn:adj}\cO_X(K_X+D+\Delta)\big|_{D}\cong \cO_D(K_{D}+\Delta \big|_D).\end{equation} If either $X$ or $D$ are singular, Equation (\ref{eqn:adj}) may no longer hold. However, there is a canonical correction term given by the \emph{different}. 

More precisely, suppose $(X,D+\Delta)$ is lc and let $\nu:D^{n} \to D$ be the normalization of $D$. 
Since $K_X+D+\Delta$ is $\mathbb{Q}$-Cartier, for $m$ divisible enough one can compare
$(m(K_X+D+\Delta))\big|_{D^n}$ and $m(K_{D^n}+\Delta\big|_{D})$.
Using this, one can define an effective
divisor $\Diff_{D^n}(\Delta)$, such that \begin{equation}\label{eqn:diff1}K_{D^n}+\Delta\big|_{D} + \Diff_{D^n}(\Delta) \sim_\mathbb{Q} (K_X+D+\Delta)\big|_{D^n}.\end{equation}
We will see that $\Diff_{D^n}(\Delta)$ is an actual effective divisor on $D^n$, not just a divisor class.

There are two equivalent definitions for $\Diff_{D^n}(\Delta)$ and both will be useful in the sequel. 

\subsubsection{The different, $1^{st}$ definition} Writing $D$ as a sum of its irreducible components and applying Equation \ref{eqn:diff1} to each component, one can see that it suffices to define $\Diff_{D^n}(\Delta)$ in the case where $D$ is irreducible. Consider then $p:Y \to X$ a log-resolution of $(X,D+\Delta)$, and let $T:=p^{-1}_*D$. Then we have the following commutative diagram.
$$\xymatrix{T \ar[r]^i \ar[d]_{p_T} & Y \ar[d]^p \\ D^n \ar[r]_j &X }$$
On $Y$, we can write $p^*(K_X+D+\Delta)= K_Y+T+F$ where $F$ is a $\mathbb{Q}$-divisor. Then $i^*p^*(K_X+D+\Delta)=K_T+F\big|_{T}$. Now we can define the different as
$$\Diff_{D^n}(\Delta):=(p_T)_*(F\big|_{T}).$$
We remark that this definition only depends on the points of codimension at most 2 on $X$, so it suffices to consider the case in which $X$ is a surface. Then from \cite{Kollarsingmmp}*{Proposition 2.35} the different satisfies the following properties:
\begin{enumerate}
    \item $\Diff_{D^n}(\Delta)$ does not depend on the choice of a log resolution, and
    \item $\Diff_{D^n}(\Delta)$ is an effective divisor.
\end{enumerate}
\subsubsection{The different, $2^{nd}$ definition} One can also define the different using the Poincaré residue map. We review the definition below. See \cite{Kollarsingmmp}*{Definition 4.2} for more details. Let $Z \subseteq X$ to be the union of $\supp \Delta \cap D$ with the closed subset where either $\supp(D)$ or $X$ are singular. Then on $X\smallsetminus Z$ there exists a \emph{canonical} isomorphism $$\cR:\omega_{X}(D)\big|_{(D \smallsetminus Z)}\to \omega_{D \smallsetminus Z}$$
given by the Poincar\'e residue map. For $m$ divisible enough so that $\omega_{X}^{[m]}(m(D+\Delta))$ is Cartier, consider the $m^{th}$ tensor power $\cR^{\otimes m}$. Since the normalization morphism $\nu : D^n \to D$ is an isomorphism on the locus where $D$ is smooth and since $\Delta|_{D \setminus Z} = 0$, $\cR^{\otimes m}$ pulls back to a rational section of
$$\cH om_{D^n}(\nu^*(\omega_{X}^{[m]}(m(D+\Delta))), \omega_{D^n}^{[m]})$$
which we also denote $\cR^{\otimes m}$. Now $\omega_{D^n}^{[m]}$ is Cartier in codimension 1 so the rational section $\cR^{\otimes m}$ defines a divisor $\Delta_m^\circ$ in codimension $1$ on $D^n$ and we denote its closure by $\Delta_m$. Then we can define the different by $$\Diff_{D^n}(D'):=\frac{1}{m}\Delta_m.$$

\subsection{Polarized orbispaces and Koll\'ar families of $\mathbb{Q}$-line bundles} 

In this subsection, we recall the definitions from \cite{AH}. All our stacks are assumed to be of finite type over a field $k$ unless otherwise noted.

\begin{Def}[\cite{AH}*{Definition 2.3.1}]
A \emph{cyclotomic stack} is a separated DM stack $\cX$ such that the stabilizers of the points of $\cX$ are finite cyclic groups.
\end{Def}
An important example of a cyclotomic stack is the \emph{weighted projective stack} $\cP(\rho_1, \ldots, \rho_n)$, defined as the stack quotient $\left[ (\mathbb{A}^n \setminus 0) / \mathbb{G}_m\right]$
where $\mathbb{G}_m$ acts on the $i^{th}$ coordinate of $\mathbb{A}^n$ by weight $\rho_i>0$. Moreover, any closed substack of a cyclotomic stack is cyclotomic. 

We will consider \emph{polarized orbispaces}, which are cyclotomic stacks analogous to projective varieties (see
\cite{AH}*{Definitions 2.3.11, 2.4.1 and 4.1.1}):
\begin{Def}\label{Def:polorbitspace}
 Let $f:\mathcal{X} \to B$ be a flat proper equi-dimensional morphism from a cyclotomic stack to a scheme. Assume that each fiber of $f$ is generically an algebraic space. Let $\pi:\mathcal{X} \to X$ be the coarse moduli space and $f_X:X \to B$ the induced induced map.
 A \emph{polarizing line bundle} is a line bundle $\mathcal{L}$ on $\mathcal{X}$ such that:
 \begin{enumerate}[(i)]
 \item For every geometric point $\xi \in \mathcal{X}(\operatorname{Spec}(K))$, the action of $\operatorname{Aut}(\xi)$ on the fiber of $\mathcal{L}$
 is effective, and
\item There is an $f_X$-ample line bundle $M$ on $X$ and an $N > 0$ such that $\mathcal{L}^N \cong \pi^*M$.
 \end{enumerate}
 A pair $(\mathcal{X} \to B, \mathcal{L})$ as above is a
 \emph{polarized orbispace}.
\end{Def}
\begin{Remark}We do not require the fibers of $f$ to be connected.\end{Remark}

Note that a weighted projective stack with the line bundle $\cO(1)$ is a polarized orbispace and by \cite{AH}*{Corollary 2.4.4}, any polarized
orbispace is a closed substack of a weighted projective stack.

Now one can define a category fibered in groupoids $\cO rb^\cL$ as follows. The objects $|\cO rb^\cL(S)|$ over a scheme $S$ are polarized orbispaces $(\cX \to B, \cL)$.
A morphism of $(\pi:\cX \to B, \cL) \to (\pi':\cX' \to B', \cL')$ over a map $g:B \to B'$ consists of a morphism $f:\cX \to \cX'$ and an isomorphism $\phi:f^*\cL' \to \cL$,  such that the following diagram is cartesian
$$ \xymatrix{\cX \ar[r]_f \ar[d]_\pi & \cX' \ar[d]^{\pi'} \\ B \ar[r]^g & B'.}$$
\begin{Teo}[\cite{AH}*{Proposition 4.2.1}] The stack $\cO rb^\cL$ is algebraic and locally of finite type.
\end{Teo}
\begin{Remark} Contrary to our conventions, we do not claim that $\cO rb^\cL$ is of finite type. \end{Remark}

For our purposes the relevance of polarized orbispaces lies in their relation with Koll\'ar families of $\mathbb{Q}$-line bundles:
\begin{Def}[\cite{AH}*{Definition 5.2.1}]\label{Def:Koll:family} A \emph{Koll\'ar family of $\mathbb{Q}$-line bundles} is the data of a pair $(f:X \to B,F)$ consisting of a morphism of schemes $f$ and a coherent sheaf $F$ on $X$ satisfying the following conditions:
\begin{enumerate}
    \item $f$ is flat with fibers that are reduced, $S_2$, and of equi-dimension $n$;
    \item For every fiber $X_b$, the restriction $F\big|_{X_b}$ is reflexive of rank 1;
    \item For every $n$ the formation of $F^{[n]}$ commutes with base change for maps $B' \to B$, and
    \item For each $X_b$, there is an $N_b$ divisible enough such that $F^{[N_b]}\big|_{X_b}$ is a line bundle.
\end{enumerate}
\end{Def}
Points (1), (2) and (4) do not pose any major difficulty as they are fiberwise conditions. However, point (3) is difficult to check and is not automatic (see \cite{AK16}). In \cite{AH}, Abramovich and Hassett give a stack theoretic characterization of families satisfying condition (3) which we now review.

First, observe that given a Koll\'ar family of $\mathbb{Q}$-line bundles $(f:X \to B, F)$, one can consider the variety $\cP(F):=\spec_{\cO_X}(\bigoplus_{n\in \mathbb{Z}} F^{[n]})$ lying over $B$. There is a natural action of $\bG_m$ over $B$ induced by the grading and taking the quotient $\cX_F:=[\cP(F)/\bG_m]$
gives a cyclotomic stack which is \emph{flat} over $B$ \cite{AH}*{Proposition 5.1.4}. Moreover, the fibers of $g$ are reduced and $S_2$ with trivial stabilizers in codimension one and there is a line bundle $\cO(1)$ on $\cX_F$ making $(g:\cX_F\to B,\cO(1))$ into a polarized orbispace. Abramovich and Hassett show that this coresponence can be reversed:

\begin{Teo}[\cite{AH}*{Section 5}]\label{Teo:AH} Consider $(f:\cX \to B,\cL)$ a polarized orbispace. Assume that for every $b \in B$ the fiber $\cX_b$ is reduced and $S_2$ with trivial stabilizers in codimension one. Let $p:\cX \to X$ be the coarse moduli space. Then $(X \to B,p_*(\cL))$ is a Koll\'ar family of $\mathbb{Q}$-line bundles.
\end{Teo}

In particular, consider a Koll\'ar family of $\mathbb{Q}$-line bundles $(X \to \spec(A),F)$ over a local Artin ring $A$ and let $A' \to A$ be an extension of local Artin rings. Then the deformations of $(X \to \spec(A),F)$ along $\spec(A) \to \spec(A')$ which satisfy the condition (4) of Definition \ref{Def:Koll:family} are identified with the deformations of the polarized orbispace $(g : \cX_F \to \spec(A), \cO(1))$. 

\begin{Def}\label{Def:AH:family} We will say that a polarized orbispace $(\cX \to B, \cL)$ satisfying the assumptions of Theorem \ref{Teo:AH} is an \emph{Abramovich-Hassett (or AH) family}. Given a Koll\'ar family of $\mathbb{Q}$-line bundles $(X \to B,F)$, we will call $(\cX_F \to B,\cO(1))$ the \emph{associated AH family}.
\end{Def}

Finally, \cite{AH} also proves the existence of a locally of finite type (but not necessarily of finite type) algebraic stack which parametrizes \emph{canonically} polarized orbispaces \cite{AH}*{Definition 6.1.1}).

\begin{Def}Following \cite{AH}, we define the moduli space parametrizing AH families of canonically polarized orbispaces with at worst slc singularities by $\cK^\omega_{\text{slc}}$. Furthermore, we denote by
$\cK_{n,v}^\omega \subseteq \cK^\omega_{\text{slc}}$ the substack parametrizing those polarized orbispaces of dimension $n$ and volume $v$.\end{Def}

\section{The moduli space of twisted stable pairs}\label{Section:moduli:TSP}

The goal of this section is to present a definition of a family of stable pairs over an arbitrary base, using polarized orbispaces. We then construct an algebraic stack $\cM_{n,v}$ of these \emph{twisted stable pairs}.

We start with the usual definition of a stable pair:
\begin{Def}
An slc pair $(X,D)$ is a \emph{stable pair} if $K_X+D$ is ample.
\end{Def}

There are two obstacles one has to overcome in order to generalize this to a notion of families of stable pairs. The first is that the $\mathbb{Q}$-divisor $K_X + D$ is only defined up to rational equivalence. Moreover, the divisor $D$ itself is an actual Weil divisor rather than just a divisor class: even when $D$ is Cartier, the condition of the pair being slc is not invariant under linear equivalence. Thus one needs to find a suitable definition for a family of divisors, over an arbitrary base scheme $B$. 

To address the first point, it is natural to consider Koll\'ar families of $\mathbb{Q}$-line bundles $(X\to B,F)$ where $F$ restricts to the reflexive sheaf $\cO(K_X + D)$ along each fiber. As we saw in the previous section, this is equivalent to considering flat families of polarized orbispaces. To address the second point, we follow an idea originally due to Koll\'ar (see \cite{Kol13}*{page 21}) and used by Kov\'acs and Patakfalvi in \cite{KP}: we replace the data of $D$ with the morphism of sheaves $\cO_X(K_X) \to \cO_X(K_X+D)$. 

\subsection{The stack of pairs $\cH$} We begin by defining a category $\cH$ fibered in groupoids over $\mathrm{Sch}/k$, consisting of pairs of a polarized orbispace and a morphism of sheaves as above.

\subsubsection{The objects of $\cH$}\label{sec:obj:H} For every scheme $B$, an object of $\cH(B)$ consists of a pair $(f : \cX \to B, \phi:\omega_{\cX/B} \to \cL)$,
with $(f : \cX \to B, \cL)$ a polarized orbispace and $\phi$ a morphism such that:
\begin{enumerate}
    \item $f : \cX \to B$ is a flat family of equi-dimensional demi-normal stacks, and
    \item for every $b \in B$, $\phi_b$ is an isomorphism at the generic points and codimension one singular points of $\cX_b$;
\end{enumerate}

\begin{Remark} Note that we do not assume that the fibers of $f$ are connected. This will simplify the definition of the moduli of gluing data in Section \ref{Section:gluing}. \end{Remark} 

\subsubsection{The arrows of $\cH$}

Consider two schemes $B_1$ and $B_2$ and a morphism $a:B_1 \to B_2$. Suppose $\alpha:=(f_1:\mathcal{X}_1 \to B_1,\phi_1: \omega_{\mathcal{X}_1/B_1} \to \mathcal{L}_1)$ and $\beta:=(f_2:\mathcal{X}_2\to B_2, \phi_2:\omega_{\mathcal{X}_2} \to \mathcal{L}_2)$ are objects in $\cH(B_1)$ and $\cH(B_2)$ respectively. An arrow $\Phi:\alpha \to \beta$ lying over $a$ is the data of a morphism $(\mu,\nu)$ of the objects $(\cX_1 \to B_1,\cL_1)$ and $(\cX_2 \to B_2,\cL_2)$ of $\cO rb^\cL$ such that following diagram is commutative: 

\begin{center}
$\xymatrix{\ar @{} [dr]
 \mu^*\omega_{\cX_2/B_2}  \ar[d]_{\mu^*(\phi_2)} \ar[r] & \omega_{\cX_1/B_1} \ar[d]^{\phi_1} \\
 \mu^*\mathcal{L}_2 \ar[r]_{\nu} & \mathcal{L}_1 . }$
\end{center}

Here $\mu^*\omega_{\cX_2/B_2} \to \omega_{\cX_1/B_1}$ is a canonical morphism, defined as follows. By Observation \ref{Oss:Gor:locus:pulls:back:to:G:locus}, there is a morphism $U(f_1) \to U(f_2)$. Since $U(f_i) \to B_i$ is Gorenstein, this induces a unique canonical isomorphism (see \cite{Conrad}*{Theorem 3.6.1})

\begin{equation}\label{eqn1:canonical}\mu^*\big|_{U(f_1)}(\omega_{U(f_2)/B_2})
\to \omega_{U(f_1)/B_1}.
\end{equation}
Letting $\iota_j:U(f_j) \to \cX_j$ be the inclusion, we have
$$
\iota_2^*(\omega_{\cX_2/B_2})=\iota_2^*(\iota_{2*}\omega_{U(f_2)/B_2})
\cong \omega_{U(f_2)/B_2},
$$
and so Equation \ref{eqn1:canonical} induces a map
$$
\iota_1^*\mu^*(\omega_{\cX_2/B_2}) \cong \mu^*\big|_{U(f_1)}(\iota_2^*(\omega_{\cX_2/B_2})) \cong \mu^*\big|_{U(f_1)}(\omega_{U(f_2)/B_2}) \to \omega_{U(f_1)/B_1}.
$$
Then the push pull adjunction gives the map $\mu^*(\omega_{\cX_2/B_2}) \to (\iota_1)_*( \omega_{U(f_1)/B_1})= \omega_{\cX_1/B_1}$.

\begin{theorem}
\label{thm:H:algebraic} The category $\cH$ is an algebraic stack locally of finite type. 
\end{theorem} 

\begin{proof} Consider the algebraic stack $\cO rb^\cL$, let $(\sX,\sL) \to \cO rb^\cL$ be the universal polarized orbispace, and assume that $\sX \to \cO rb^\cL$ has relative dimension $n$. The locus $\cO^{(1)} \to \cO rb^\cL$ where the fibers are $S_2$ and reduced is open by \cite{EGAIV}*{Théorème 12.2.4}. Furthermore, the condition of having at worst nodal singularities is open, so there is an open substack $\cO^{(2)} \hookrightarrow \cO^{(1)}$ where the fibers of $\sX^{(2)}:=\sX \times_{\cO rb^\cL} \cO^{(2)} \to \cO^{(2)}$ are demi-normal. 

Now over $\cO^{(2)}$, we can consider the relative canonical sheaf $\omega_{\sX^{(2)}/\cO^{(2)}}$ which is reflexive and equal to $\iota_*(\omega_{U/\cO^{(2)}})$ where $U$ is the relative Gorenstein locus of $\sX^{(2)} \to \cO^{(2)}$. Denoting by $\sL^{(2)}$ the pullback of $\sL$ to $\sX^{(2)}$, we have that the Hom-stack $\cH' := \underline{\Hom}_{\cO^{(2)}}(\omega_{\sX^{(2)}/\cO^{(2)}},\sL^{(2)})$ is algebraic and locally of finite type by \cite{hom}*{Proposition 2.1.3}. We will check that $\cH$ is a substack of $\cH'$ by identifying $\cH'$ with the stack of pairs $(f : \cX \to B, \phi: \omega_{\cX/B} \to \cL)$. That is, for every scheme $B$ there is an equivalence of categories
$$
\operatorname{Hom}(B, \cH') \to \{(f:\cX\to B, \phi : \omega_{\cX/B} \to \cL)\}
$$
where $(\cX \to B, \cL)$ is induced by a morphism $B \to \cO^{(2)}$.

Indeed given a morphism $B \to \cH'$, we can first compose it with the natural map $\cH' \to \cO^{(2)}$ to get a map $B \to \cO^{(2)}$ and thus a family of demi-normal polarized orbispaces $(\cX \to B, \cL)$. Then the universal property of the Hom stack identifies the category of liftings as in the dotted arrow below
$$
\xymatrix{ & \sH \ar[d]\\ B \ar@{.>}[ur] \ar[r]^{\pi}& \cO^{(2)}}
$$
with maps $\phi_0: \pi^*(\omega_{\sX^{(2)}/\cO^{(2)}}) \to \cL$. Now since $\cL$ is a line bundle, $\phi_0$ factors uniquely through the reflexive
hull of the source, so we obtain $\phi_0^{[1]} : \pi^*(\omega_{\sX^{(2)}/\cO^{(2)}})^{[1]} \to \cL$. Notice that $$\pi^*(\omega_{\sX^{(2)}/\cO^{(2)}})^{[1]} \cong \omega_{\cX /B}$$ so setting $\phi:= \phi_0^{[1]}$ gives a pair. We leave it to the reader to check this is an equivalence of categories.

For any $B$ and any object $(f : \cX \to B, \omega_{\cX/B} \to \cL) \in \cH'(B)$, condition (1) is satisfied by construction since $(f : \cX \to B, \cL)$ is an object of $\cO^{(2)}(B)$. We must show that (2) imposes an algebraic condition. In fact, we will show that it cuts out $\cH$ as an open substack of $\cH'$. Let $(\sX^{(3)}, \sL^{(3)})$ denote the pullback of the universal polarized orbispace to $\cH'$ and let $\Psi : \omega_{\sX^{(3)}/\cH'} \to \sL^{(3)}$ be the universal morphism. 

Since $\pi:\sX^{(3)} \to \cH'$ is Gorenstein in codimension one, $\omega_{\pi}$ is a line bundle in codimension one and a morphism of line bundles is an isomorphism if and only if it is surjective. Thus we need to show that requiring $\Psi$ to be surjective at generic points and at the codimension one singular points of the fibers of $\pi$ is an open condition. 

We begin with the generic points. Consider $\sC:=\operatorname{Supp}(\operatorname{Coker}(\Psi))$. Since $\sC \to \sX^{(3)}$ is a closed embedding and $\sX^{(3)} \to \cH'$ is proper, $\cC \to \cH'$ is proper. We need to show that the locus where the fibers of $\cC\to \cH'$ have dimension at most $(n-1)$ is open. Since $\cC$ and $\cH'$ are tame $DM$ stack, and formation of coarse moduli spaces commutes with base change, it suffices to show that the locus where the fibers of the coarse moduli map $C \to H'$ have dimension at most $(n -1)$ is open. This follows from upper-semicontinuity of fiber dimension. 

Finally we consider the codimension one singular points. Let $\sS \subset \sX^{(3)}$ be the singular locus of $\pi: \sX^{(3)} \to \cH'$. It is a closed substack of $\sX^{(3)}$ so that the map to $\cH'$ is proper. Let $\Psi_\sS$ be the restriction of $\Psi$ to $\sS$ and $\cC_S$ be the support of its cokernel. Then it suffices to show that the locus where the fibers of $\cC_S \to \cH'$ have dimension at most $(n-2)$ is open. Since $\cC_S$ is a closed substack of $\sS$, it is proper and the result again follows from upper-semicontinuity of fiber dimension applied to the coarse map. 

Thus $\cH$ is an open substack of the algebraic stack $\cH'$. 
\end{proof} 

\subsection{The family of divisors $\cD$} Next we produce a family of divisors from the data of a pair $(f : \cX \to B, \phi : \omega_{\cX/B} \to \cL) \in \cH(B)$, and study its properties. 

\begin{Lemma}\label{Lemma:phi:gives:an:ideal}
Suppose $(f:\cX \to B, \phi:\omega_{\cX/B} \to \cL)$ is an object of $\cH(B)$. Then the morphism $\phi \otimes \cL^{-1}:\omega_{\cX/B}\otimes \cL^{-1} \to \cO_{\cX}$ is injective.
\end{Lemma}
\begin{proof}
 The statement is étale local so we can replace $\cX$ with an étale cover and assume it is a scheme. The sheaf $\omega_{\cX/B}\otimes \cL^{-1}$ is reflexive by \cite{YN18}*{Proposition 5.6}. Therefore by \cite{HK}*{Proposition 3.5}, if $j:U(f) \to \cX$ is the open embedding of the Gorenstein locus, $$\omega_{\cX/B}\otimes \cL^{-1} \cong j_*j^*(\omega_{\cX/B}\otimes \cL^{-1}) \text{ and }\cO_{\cX} \cong j_*j^*(\cO_X).$$
 This means that for every open subset $V \subseteq \cX$ the restriction morphism $(\omega_{\cX/B}\otimes \cL^{-1})(V) \to
(\omega_{\cX/B}\otimes \cL^{-1})(V \cap U(f)$ is an isomorphism.
Therefore, it suffices to show that $(\omega_{\cX/B}\otimes \cL^{-1})(V \cap U(f)) \to \cO_{\cX}(V \cap U(f))$ is injective. Namely, we can assume without loss of generality that $\omega_{\cX/B}$ is a line bundle.

Now the statement is local so suppose that $B=\spec(A)$, $\cX=\spec(R)$ with $g:\spec(R) \to \spec(A)$ where $R$ is a flat $A$-module, and $\omega_{\cX/B} \otimes \cL^{-1} \cong \cO_{\cX}$ and denote $\phi \otimes \cL^{-1}$ by $\psi$. Then $\psi:R \to R$ is an $R$-module homomorphism which is necessarily multiplication by some $a \in R$. Our goal is to show that $a$ is not a zero divisor.

From the commutativity condition on morphisms and point (2) in the definition of objects of $\cH$, the element $a$ is not a zero divisor when restricted to each fiber of $\cX \to B$.
For any ring $C$ and $C$-module $M$, we will denote by $\Ass_C(M)$ the associated primes of $M$, and by $\Div_C (M)$ the set of zero divisors for M. It is essential now to recall that $\Div_C(M)= \bigcup_{P \in \Ass_C(M)}P$.
Then the following chain of implications finishes the proof:

\begin{align*}
    & a \notin \bigcup_{p \in \spec(A)} \Div _R(R/pR) \text{ }  \Rightarrow\text{ }  a \notin \bigcup_{p \in \spec(A),\text{ } q \in \Ass_R(R/pR)} q \text{ }\Rightarrow \\ &\text{ }\Rightarrow a \notin \bigcup_{p \in \Ass_A(A), \text{ } q \in \Ass_R(R/pR)} q \text{ } \xRightarrow{(\ast)}\text{ } a \notin \bigcup_{q \in \Ass_R(R)}q \text{ }\Rightarrow\text{ } a \notin \Div_R(R),
\end{align*}
where the arrow labelled with $(\ast)$ follows from \cite{Mat89}*{Theorem 23.2 (ii)} and its proof. 
\end{proof}

\begin{Notation}\label{Not:divisor} The ideal sheaf given by taking the image of $\phi \otimes \cL^{-1}$ in Lemma \ref{Lemma:phi:gives:an:ideal} will be denoted by $\cO_{\cX}(-\cD_\phi)$ and the resulting closed substack will be denoted by $\cD_\phi \subseteq \cX $. Furthermore, we will drop the subscript $\phi$ when there is no risk of confusion.
\end{Notation}

\begin{Cor}\label{Cor:D:is:a:FLAT:family}
Let $(f : \cX \to B, \phi : \omega_{\cX/B} \to \cL)$ be an object of $\cH(B)$ and suppose that for each $b \in B$, $\cX_b$ is Du Bois. Then $\cD_\phi \to B$ is flat and commutes with basechange. 
\end{Cor}

\begin{proof} We will drop the subscript $\phi$ in $\cD_\phi$. Let $B' \to B$ be a morphism and consider the pullback $(\cX' \to B', \phi':\omega_{\cX'/B'} \to \cL')$. Let $g:\cX' \to \cX$ be the resulting morphism.

By the exact sequence of $\operatorname{Tor}$ and using that $\cX \to B$ is flat, it suffices to show that the morphism $g^*(\cO_{\cX}(-\cD)) \to g^*\cO_{\cX} \cong \cO_{\cX'}$ is injective. Since the fibers are Du Bois, the relative canonical sheaf commutes with base change by \cite{KK18}*{Corollary 1.5}, so $g^*(\omega_{\cX/B}) \cong \omega_{\cX'/B'}$. Thus
$$
g^*(\cO_{\cX}(-\cD)) \cong g^*(\omega_{\cX/B} \otimes \cL^{-1})\cong g^*(\omega_{\cX/B}) \otimes g^*(\cL)^{-1} \cong \omega_{\cX'/B'} \otimes (\cL')^{-1}.
$$
By the commutativity condition in the definition of $\cH$, the composition $\omega_{\cX'/B'} \otimes (\cL')^{-1} \cong g^*(\cO_{\cX}(-\cD)) \to \cO_{\cX'}$ is the map $\phi' \otimes (\cL')^{-1}$ which is injective by Lemma \ref{Lemma:phi:gives:an:ideal} and identifies $g^*(\cO_{\cX}(-\cD))$ with $\cO_{\cX'}(-\cD')$. 
\end{proof}

\subsection{The stack $\cM_{n,v}$} We are now ready to introduce the stack of twisted stable pairs $\mathcal{M}_{n,v}$ as a category fibered in groupoids over $\mathrm{Sch}/k$. 

\begin{Def}\label{Def:moduli:functor}  For $B$ a scheme, an object of $\mathcal{M}_{n,v}(B)$ consists of a pair $(f:\cX \to B, \phi:\omega_{\mathcal{X}/B} \to \mathcal{L}) \in \cH(B)$ such that 
\begin{enumerate}
    \item for every $b \in B$, $(\cX_b,\cD_b)$ is slc, and
    \item for every $b \in B$, the volume of $(\cX_b,\cD_b)$ is $v$.
\end{enumerate}
Morphisms are given by morphisms in $\cH$. We will call an object of $\cM_{n,v}$ over $B$, a \emph{twisted stable pair} over $B$.
\end{Def}

\begin{Remark} Despite the fact that $\cM_{n,v}$ is defined as a full subcategory of $\cH$, it is not immediate that it is a substack. The issue is that condition (1) in Definition \ref{Def:moduli:functor} is not compatible with base change a priori. However, by \cite{Kollarsingmmp}*{Theorem 5.14} the fibers
$\cX_b$ of a twisted stable pair are Du Bois. Therefore, by Corollary \ref{Cor:D:is:a:FLAT:family}, $\cD$ is flat and commutes with base change so $\cM_{n,v}$ is a category fibered in groupoids. 
\end{Remark} 

\begin{Teo}\label{Teo:we:have:alg:stack}The stack $\cM_{n,v}$ is algebraic and locally of finite type. 
\end{Teo}

\begin{proof} Since the stack of pairs  $\cH$ is algebraic by Theorem \ref{thm:H:algebraic}, it suffices to show that $\cM_{n,v}$ is an open substack. Thus we will check that (1) and (2) in Definition \ref{Def:moduli:functor} are open conditions. 

\emph{Condition (1).} Consider a family of pairs $(f:\cX\to B, \cL, \phi:\omega_{\mathcal{X}/B} \to \mathcal{L}) \in \cH$.  We need to show that the locus in $B$ where $(\cX_b, \cD_b)$ is slc is open. It suffices then to show that it is constructible, and stable under generalization. Our proof is
inspired by \cite{kol1}*{Lemma 4.48} and \cite{AH}*{Lemma A.2.1}. See also \cite{KP}*{Lemma 5.10}. 

We begin with some initial reductions. By Corollary \ref{Cor:D:is:a:FLAT:family}, $\cD \to B$ is flat for objects of $\cM_{n,v}$ so $\cM_{n,v}$ is contained in the substack $\cH_1 \subset \cH$ where $\cD \to B$ is flat. By \cite{EGAIV}*{Theorem 11.3.1}, $\cH_1 \subset \cH$ is open. Furthermore, $\cM_{n,v}$ is contained in the substack $\cH_2 \subset \cH_1$ where the fibers of $\cD \to B$ are reduced. Since $\cD \to B$ is flat and proper, then $\cH_2 \subset \cH_1$ is open by \cite{EGAIV}*{Theorem 12.2.1}. Thus without loss of generality, we may suppose that $\cD \to B$ is flat with reduced fibers. 

Next we show that the locus in $B$ where $(\cX_b, \cD_b)$ is slc is constructible. For that, we can assume that $B$ is irreducible. Let $\eta$ be the generic point of $B$ and consider the generic fiber $(\cX_\eta,\omega_{\cX_\eta/\eta} \to \cL_\eta)$. Let $\cY_\eta \to \cX_\eta$ of $(\cX_\eta,\cD_\eta)$ be a log-resolution, which exists for DM stacks by functorial resolution of singularities (see for example \cite{Wlo05}). Since $B$ is locally of finite type over a field,  this resolution can be spread out to give a simultaneous log-resolution $\pi:\cY \to \cX\big|_{U}$ of $(\cX\big|_{U},\cD\big|_{U})$ for a suitable open subset $U \subseteq B$ using the method of \cite{kol1}*{Lemma 4.48}.
Then we have $\omega_{\cY/U} \otimes \cO_\cY(\pi_*^{-1}(\cD\big|_{U}))\cong \pi^*(\cL\big|_{U}) \otimes \cO_{\cY}(\sum d_i \Delta_i) $ for appropriate irreducible $\pi$-exceptional
divisors $\Delta_i$ and coefficients $d_i$. 
Then for every $b \in U$, the pair $(\cX_b,\cD_b)$ is slc if and only if $d_i \ge -1$ for every $i$ which is a constructible condition by Chevalley's Theorem. 

To show that being slc is stable under generalization, we proceed as in \cite{KP}*{Lemma 5.10}. Consider a family of demi-normal orbifolds $\cX \to \spec(R)$ over a DVR $R$ and assume we have a divisor $\cD$ such that $K_\cX+\cD$ is $\mathbb{Q}$-Cartier. After passing to an \'etale cover, we may assume $(\cX, \cD)$ are schemes. If the closed fiber $(\cX_p,\cD_p)$ is slc, then by inversion of adjunction (\cite{Pat16}*{Corollary 2.11}) we conclude that the pair $(\cX,\cD) $ is slc so the condition is stable under generalization. 

\emph{Condition (2).} We need to show that the locus in $\cH$ where the fibers have volume $v$ is open for every integer $v$. Note that since $\cL$ is a line bundle, the volume must be integral. Since $\cH$ is locally of finite type, we can restrict to an open connected substack $\cH_0$ of finite type. Let $(\sX_0, \sL_0) \to \cH_0$ be the pullback of the universal polarized orbispace and let $\pi: \sX_0 \to X_0$ be the relative coarse moduli space of $X_0 \to \cH_0$. Since $X_0$ is of finite type, there exists an $m$ such that $\sL_0^{\otimes m} = \pi^*L$ for some line bundle $L$ on $X_0$. It suffices to know that the self intersection $(L\big|_{(X_0)_b})^{n}$ is locally constant, which follows from the flatness of $X_0 \to \cH_0$.
\end{proof}

Observe that the stack $\cM_{n,v}$ comes with a morphism to $\cO rb^\cL$. From \cite{AH}*{Corollary 5.3.7}, there is an open embedding $\cK^\cL \to \cO rb^\cL$ parametrizing cyclotomic stacks which are $S_2$ and have no stabilizers in codimension one.

\begin{Def}\label{Def:KSBA:component}
 Let $\cK_{n,v}:= \cK^\cL \times_{\cO rb^\cL} \cM_{n,v} \to \cM_{n,v}$ be the open substack parametrizing orbispaces which have no stabilizers in codimension one. We call $\cK_{n,v}$ the \emph{Koll\'ar component} of $\cM_{n,v}$ and define the twisted stable pairs parametrized by $\cK_{n,v}$ the \emph{AH twisted stable pairs}. 
 \end{Def}

By Theorem \ref{Teo:AH}, the objects of $\cK_{n,v}$ are the twisted stable pairs $(f : \cX \to B, \phi : \omega_\cX \to \cL)$ such that $(f : \cX \to B, \cL)$ is the associated AH family of a Koll\'ar family of $\bQ$-line bundles. That is, they are twisted stable pairs that are also AH families as in Definition \ref{Def:AH:family}. 
 
\begin{Notation}
Let $\cK_{n,v}^c$ denote the open substack of $\cK_{n,v}$ parametrizing connected twisted stable pairs.
\end{Notation}
\begin{Teo}\label{Teo:the:KSBA:stack:is:proper:DM}
 The stacks $\cK_{n,v}^c$  and $\cK_{n,v}$ are proper  and DM.
\end{Teo}
\begin{proof}It suffices to prove the desired claim for $\cK_{n,v}^c$. We divide the proof into three steps.

\emph{$\cK_{n,v}^c$ is of finite type}: This follows from the results \cite{HMX} and \cite{Kol08} (see also \cite{Kol11}). By \cite{HMX}*{Theorem 1.1}, there is a projective morphism of quasi-projective varieties $X \to B$, with a divisor $D \subseteq X$ such that every stable pair of dimension $n$ and volume $v$ appears as $(X_b,D_b)$ for a certain $b \in B$. Up to replacing $B$ with a resolution, we may assume that $B$ is smooth. By taking the flattening stratification, we can assume that $X \to B$ is flat. Using boundedness and proceeding as in \cite{Kol08}*{Corollary 25}, we can stratify further so that $(K_{X_b}+D_b)^{[m]}$ commutes with base change for every $m \in \mathbb{Z}$. In particular, we can consider the associated AH family $(f:\cX \to B, \cL)$. On the locus $\cU$ where $f$ is Gorenstein and has trivial stabilizers, we have a morphism $\omega_{f|_{\cU}} \to \cL\big|_{\cU}$ which extends uniquely to a morphism $\omega_f \to \cL$ by Lemma \ref{Lemma:unique:arrow:reflexive:hull}. Since every point of $\cK_{n,v}^c$ corresponds to the AH stack of a fiber of $X \to B$ and $B$ is of finite type, it follows that $\cK_{n,v}^c$ is as well.

\emph{$\cK_{n,v}^c$ is proper}: The valuative criterion for properness now follows from \cite{kol1}*{Theorem 2.50 and Proposition 2.76 (2)}.

\emph{$\cK_{n,v}^c$ is DM}: Consider an AH twisted stable pair $(\cX \to \spec(k),\omega_\cX \to \cL)$, and let $X$ be the coarse space of $\cX$. Since $\cX$ has no generic stabilizers, an automorphism of $\cX$ which induces the identity on $X$ is the identity. So it suffices to know that the pair $(X,D)$ has finite automorphisms where $D=\cD_\phi^c$. This is \cite{KP}*{Proposition 5.5}. 
\end{proof}

\section{Local and global structure of twisted stable pairs}\label{Section:local:and:global:structure:of:TSP}
The goal of this section is to study the twisted stable pairs of Definition \ref{Def:moduli:functor}.
The section is divided into two subsection. In the first one we focus on the properties of an AH family, and we give conditions that are equivalent to having a morphism to an AH family. In the second section we study the local structure of a twisted stable pair in a neighbourhood of the divisor $\cD$. By a careful analysis of the case that $\cX$ is a surface, we show that the the only contribution to the different on a twisted stable pair comes from the double locus of $\cD$ (see Proposition \ref{Prop:no:different:on:stack} and Corollary \ref{Cor:different:only:from:double:locus:arbitrary:dim}).

\subsection{Global structure of AH families} In this subsection, we study the following questions. Given an AH family $\cX \to B$ over a scheme $B$ with coarse space $X$ and a morphism $f:\cY \to X$ from a DM stack $\cY$, then: \begin{enumerate}
    \item when can we lift $f$ to a map $\cY \to \cX$;
    \item when is such a lift to be an isomorphism? 
\end{enumerate} 
Among other things, this will allow us to relate our definition of a family of twisted stable pairs over the spectrum of an Artin ring with the $\mathbb{Q}$-Gorenstein deformations of Hacking in \cite{Hac04} (see Corollary \ref{Cor:comparison:Hacking}). 

\begin{Prop}\label{Prop:functor:morphism:to:the:AHstack}
Suppose $(X\to B,F)$ is a Koll\'ar family of $\mathbb{Q}$-line bundles with associated AH family $(\cX\to B,\cL)$. Let $p:\cX \to X$ be the coarse moduli space map and consider a morphism $f:Y \to X$ from a scheme $Y$. Then the groupoid $\Hom_X(Y,\cX)$ of maps lifting $f$ is equivalent to the following groupoid:
\begin{itemize}
    \item objects are given by pairs $(\cG,\phi)$ where $\cG$ is a line bundle on $Y$, and 
    $\phi:\bigoplus_{n \in \mathbb{Z}} f^*(F^{[n]}) \to \bigoplus_{n \in \mathbb{Z}} \cG^{\otimes n}$ is a homomorphism of graded $\cO_Y$-algebras, and
    \item morphisms between $(\cG_1,\phi_1) \to (\cG_2,\phi_2)$ are given by an isomorphism of line bundles $\psi:\cG_1 \to \cG_2$ such that $\phi_2= (\bigoplus_n \psi^{\otimes n}) \circ \phi_1$.
\end{itemize}

\end{Prop}
\begin{proof}
 Recall that, as a stack over $X$, we have $\cX=[\cP(F)/\mathbb{G}_m]=[\spec_{\cO_X}(\bigoplus_{n \in \mathbb{Z}} \cF^{[n]})/\mathbb{G}_m]$. Therefore for a scheme $Y$ over $X$, the groupoid $\cX(Y)$ is equivalent to the groupoid of $\mathbb{G}_m$-torsors $E \to Y$, with a $\mathbb{G}_m$-equivariant morphism $E \to \cP(F)$ over $X$. This is equivalent to a $\mathbb{G}_m$-equivariant morphism $E \to \cP(F)\times _X Y \cong \spec(\bigoplus_{n \in \mathbb{Z}} f^*(F^{[n]}))$. There exsts a unique line bundle $\cG$ on $Y$ such that the $\mathbb{G}_m$-torsor $E \to Y$ may be written as $\spec(\bigoplus_{n \in \mathbb{Z}} \cG^{\otimes n}))$. To conclude the proof it suffices to notice that $\mathbb{G}_m$-equivariant morphisms $$\spec(\bigoplus_{n \in \mathbb{Z}} \cG^{\otimes n}))\to \spec(\bigoplus_{n \in \mathbb{Z}} f^*(F^{[n]}))$$ correspond to graded $\cO_Y$-algebra homomorphisms $\bigoplus_{n \in \mathbb{Z}} f^*(F^{[n]}) \to \bigoplus_{n \in \mathbb{Z}} \cG^{\otimes n}$.
\end{proof}

\begin{Cor}\label{Cor:particular:case:lemma:functor:morphism:to:the:AHstack} Let $(p:X\to B,F)$ a Koll\'ar family of $\mathbb{Q}$-line bundles with associated AH family $(\cX\to B,\cL)$. Consider a DM stack $\cY$ with a morphism $f:\cY \to X$. Assume that there is a line bundle $\cG$ on $\cY$ and an isomorphism of graded $\cO_X$-algebras $\bigoplus_{n\in \mathbb{Z}} f_*(\cG^{\otimes n}) \cong \bigoplus_{n\in \mathbb{Z}} F^{[n]}$.
Then there is a morphism $g:\cY \to \cX$ over $X$ such that the following diagram is cartesian.
$$ \xymatrix{\cP(\cG) \ar[r] \ar[d] & \cP(F) \ar[d] \\ \cY \ar[r] & \cX}$$
\end{Cor}

\begin{proof} First observe that
$$
     \Hom \left(f^*\bigoplus_n  F^{[n]}, \bigoplus_n \cG^{\otimes n}\right)\cong\Hom \left(\bigoplus_n F^{[n]},f_* \bigoplus_n \cG^{\otimes n}\right) \cong\Hom \left(\bigoplus_n F^{[n]}, \bigoplus_n f_*(\cG^{\otimes n})\right).
$$
The graded isomorphism in the hypothesis gives an isomorphism of $\mathbb{G}_m$-torsors $$\cP(\cG) \to \spec(f^*(\bigoplus_n  F^{[n]}))\cong \cP(F) \times_X \cY.$$ The resulting morphism $\mathbb{G}_m$-equivariant map $\cP(\cG) \to \cP(F)$ induces $\cY\cong [\cP(\cG)/\mathbb{G}_m] \to [\cP(F)/\mathbb{G}_m] \cong \cX$ and the diagram is cartesian by \cite{Ols16}*{Exercise 10.F}.
\end{proof}

Note that in Corollary \ref{Cor:particular:case:lemma:functor:morphism:to:the:AHstack}, if we assume that $\cY \to X$ is the coarse space, we cannot conclude that the morphism $\cY \to \cX$ is an isomorphism. For example,  consider $\cY=B\bmu_2$ with $\cG=\cO_{\cY}$ and $\cX=X=\spec(k)$ with $\cL=\cO_X$. The problem is that $\cY$ might have more stabilizers than $\cX$. The following lemma shows that this is the only reason for the failure of $\cY \to \cX$ to be an isomorphism.

\begin{Lemma}\label{Lemma:iso:to:AH:stack:if:P(F):scheme}
 In the situation of Corollary \ref{Cor:particular:case:lemma:functor:morphism:to:the:AHstack}, suppose $\cY$ is separated with coarse moduli space $X$. Then the morphism $g:\cY \to \cX$ is the relative coarse moduli space of the map $\cY \to B\bG_m$ induced by $\cG$.
\end{Lemma}

\begin{proof} Let $\cY \to \cY' \to B\bG_m$ the relative coarse moduli space of the map $\cY \to B\bG_m$. Then $\cY' \to B\bG_m$ is representable and equips $\cY'$ with a line bundle $\cG'$ that pulls back to $\cG$. Furthermore, $f$ factors through $f' : \cY' \to X$ and we have $f'_*(\cG')^{\otimes m} = f_*(\cG)^{\otimes m} = F^{[m]}$. Thus we may apply Corollary \ref{Cor:particular:case:lemma:functor:morphism:to:the:AHstack} to obtain a morphism $g' : \cY' \to \cX$ and it suffices to prove $g'$ is an isomorphism. In particular, without loss of generality we may assume that $\cY \to B\bG_m$ is representable and we wish to prove that $g : \cY \to \cX$ is an isomorphism. 

To prove this claim, note first that the question is local over $X$, so from \cite{Ols16}*{Theorem 11.3.1} we can assume that $X=V/G$ and $\cY=[V/G]$ for a finite group $G$. Moreover, we can assume that $V=\spec(A)$ is affine. Then $\cG$ corresponds to a locally free $A$-module $M$ with a $G$-action and the coarse space of $\cG \cong [\spec( \bigoplus_{n \in \mathbb{Z}} M^{\otimes n})/G] $ is given by $$\spec\left(\bigoplus_{n \in \mathbb{Z}} ((M^{\otimes n})^G)\right) = \spec\left(\bigoplus_{n \in \mathbb{Z}}  f_*(\cG^{\otimes n})\right) \cong \spec\left(\bigoplus_{n\in \mathbb{Z}} F^{[n]}\right)=\cP(F).$$
Now, by \cite{AH}*{Proposition 2.3.10} the representability of $\cY \to B\bG_m$ implies that for every $p \in \cY$, the action of $\Aut_\cY(p)$ on $\cG_p$ is faithful. 
Then $\cP(\cG)=\spec_\cY ( \bigoplus_{n \in \mathbb{Z}} \cG^{\otimes n})$ is already an algebraic space so it is isomorphic to its coarse space: $\cP(F) \cong \cP(\cG)$. Thus, since $\cP(F) \to \cX$ is a smooth atlas and the diagram of Lemma \ref{Cor:particular:case:lemma:functor:morphism:to:the:AHstack} is cartesian, $g$ is an isomorphism.
\end{proof}

Lemma \ref{Lemma:iso:to:AH:stack:if:P(F):scheme} allows us to compare our definition of twisted stable family with $\mathbb{Q}$-Gorenstein deformations in the case where the divisor $D$ is Cartier.

\begin{Cor}\label{Cor:comparison:Hacking}
Let $(X,D)$ be an slc pair where $D$ is a Cartier divisor. Then $K_X$ is $\mathbb{Q}$-Cartier and the canonical covering stack $\cX' \to X$ is isomorphic to the AH stack $\cX$ for the $\mathbb{Q}$-line bundle $L=K_X+D$.
\end{Cor}

\begin{proof}
Consider the coarse space map $p:\cX' \to X$ and let $\cG:=K_{\cX'}+p^*D$. Then:
\begin{enumerate}
    \item $\cG$ is a line bundle with $p_*(\cG^{\otimes n})= \cO_X(K_X+D)^{[n]}$, and
    \item for every point $q \in \cX'$, the action of $\Aut_{\cX'}(q)$ on $\cG_q$ is faithful.
\end{enumerate}
Point (1) follows since both sides of the equality are reflexive and they agree on the big open set where $p$ is an isomorphism. Point (2) holds because the action of $\Aut_{\cX'}(q)$ on $(K_{\cX'})_q$ is faithful by definition, the action on $(p^*D)_q$ is trivial, and the tensor product of a trivial and a faithful representation is a faithful representation. Thus we may apply Lemma \ref{Lemma:iso:to:AH:stack:if:P(F):scheme} to conclude that $\cX \cong \cX'$.
\end{proof}

\subsection{Local structure of twisted stable pairs}\label{Subsection:different} 

The goal of this subsection is to study the local structure of $\cX$ along the divisor $\cD$. First, we prove Theorem \ref{thm:intro2} (see Proposition \ref{Prop:no:different:on:stack} and Corollary \ref{Cor:different:only:from:double:locus:arbitrary:dim}) which says that on a twisted stable pair the singularities of $\cX$ do not contribute to the different. Then we explore the relationship between the stack structure on $\cD$ and the different on the coarse space of the pair (Lemmas \ref{Lemma:relation:singularities:different} and \ref{Lemma:AH:stack:iso:along:nodes}). 

We start with the local notion of a twisted pair: 

\begin{Def} A \emph{twisted pair} is a pair $(\cX, \phi:\omega_{\cX} \to \cL)$ where $\cX$ is an open substack of an $\cX'$ with $(\cX' \to \spec(k),\phi':\omega_{\cX'} \to \cL)$ a twisted stable pair and $\phi$ pulled back from $\phi'$.
\end{Def}

\begin{Prop}\label{Prop:no:different:on:stack} Let $(f:\cX \to B, \phi:\omega_f \to \cL)$ be a twisted surface pair. Then $\cL\big|_{\cD_\phi} \cong \omega_{\cD_\phi /B}$.
\end{Prop}

\begin{proof}
We drop the subscript $\phi$ in $\cD_\phi$. The proof follows closely \cite{Kollarsingmmp}*{Pag 152}.

Let $\iota:\cD \to \cX$ denote the natural closed embedding. From Grothendieck duality, since $f$ is Cohen-Macauley and $\cD \to B$ is Gorenstein, we have $\iota_*\omega_{\cD/B} \cong \cE xt^1(\iota_* \cO_\cD,\omega_f)$. Consider the exact
sequence $$0 \to \cO_\cX(-\cD) \to \cO_\cX\to \iota_*\cO_\cD \to 0.$$  Taking $\cH om_\cX(-,\omega_f)$ induces a morphism
$$ \cH om(\cO_\cX, \omega_f) \to\cH om(\cO_\cX(-\cD), \omega_f) \to \cE xt^1 (\iota_* \cO_\cD, \omega_f) \to \cE xt^1(\cO_\cX, \omega_f) = 0$$
where the last term is 0 since $\cO_\cX$ is locally free.
But from Lemma \ref{Lemma:phi:gives:an:ideal}, the ideal $\cO_\cX(-\cD)$ is defined as $\omega_f \otimes \cL^{-1}$, therefore
$$\cH om(\cO_\cX(-\cD), \omega_f) \cong \cH om(\omega_f \otimes \cL^{-1}, \omega_f)
\cong \cH om(\omega_f, \omega_f)\otimes \cL.$$

Now from Lemma \ref{Lemma:unique:arrow:reflexive:hull}, taking $U$ the $f$-Gorenstein locus, we have $\cH om(\omega_f, \omega_f) \cong \cO_\cX$. Therefore the sequence above gives a surjective morphism $\cL \to \iota_*(\omega_{\cD/B})$. This induces a surjective morphism $\iota^*(\cL) \to \omega_{\cD/B}$, but a surjective morphism of line bundles is an isomorphism.
\end{proof}

\begin{Cor}\label{Cor:different:only:from:double:locus:arbitrary:dim}
Let $(\cX\to \spec(k), \omega_\cX \to \cL)$ be a twisted stable pair of any dimension. Then $\cL\big|_{\cD} \cong \omega_{\cD}$. Moreover, when $\cD$ is $S_2$, it is an slc canonically polarized orbispace.\end{Cor}
\begin{proof} By \cite{Kov18}*{Lemma 3.7.5}, the canonical sheaf $\omega_Z$ is $S_2$ on any excellent scheme $Z$ admitting a dualizing complex. Furthermore, by \cite{BH93}*{Proposition 1.2.16}, the proeprty of being $S_2$ can be checked after finite \'etale base change. Therefore $\omega_{\cD}$ is an $S_2$ sheaf on $\cD$. Now, we may apply Proposition \ref{Prop:no:different:on:stack} on the codimension $2$ points of $\cX$ to see that $\omega_{\cD}$ agrees with $\cL\big|_{\cD}$ in codimension $1$ on $\cD$. Since both sheaves are $S_2$, we conclude that $\cL\big|_{\cD} \cong \omega_{\cD}$. 

If $\cD$ is $S_2$, since it is nodal in codimension $1$, it is demi-normal. In particular by inversion of adjunction (\cite{Pat16}*{Lemma 2.11}) $\cD$ is slc. Finally, $\omega_{\cD}$ is a polarizing line bundle since it is the pullback of one by a closed embedding.
\end{proof} 

\begin{Remark} Usually when one considers questions about adjunction for stable pairs, one works with the normalization of the divisor $n : \cD^n \to \cD$. By Corollary \ref{Cor:different:only:from:double:locus:arbitrary:dim}, we have that $(\iota \circ n)^*\cL = n^*\omega_{\cD}$. Since $\cD$ is nodal in codimension $1$, this latter sheaf is isomorphic to $\omega_{\cD}(G)$ where $G$ is the double locus of $\cD$. Therefore, we may interpret the corollary as stating that the only contribution to the different on $\cD^n$ comes from the double locus of $\cD$. 
\end{Remark}

Note that in general, we have an adjunction morphism $\cM_{n,v} \to \cO rb^\cL$ given by $$(\cX \to B, \phi : \omega_{\cX} \to \cL) \mapsto (\cD_{\phi}, \cL\big|_{\cD_{\phi}}).$$ Indeed $\cD_\phi \to B$ is flat with reduced fibers (Corollary \ref{Cor:D:is:a:FLAT:family}) and $\cL\big|_{\cD}$ is a polarizing line bundle. Consider the locus $\cU \subseteq \cM_{n,v}$ where $\cD_\phi \to \cM_{n,v}$ is $S_2$ which is open by \cite{EGAIV}*{Théorème 12.1.6} and the fact that the morphism is closed. Over $\cU$, Corollary \ref{Cor:different:only:from:double:locus:arbitrary:dim} gives us that $(\cD_\phi \to B, \cL\big|_{\cD_{\phi}})$ is a family of canonically polarized slc orbispaces of dimension $(n-1)$ and so over this locus, the morphism above lands in the open substack of $\cO rb^{\omega}$ of those orbispaces with slc singularities. 

The following corollary follows from Proposition \ref{Prop:no:different:on:stack} and \cite{Kollarsingmmp}*{Proposition 2.35}:
\begin{Cor}
Let $(\cX,\cD)$ be a twisted surface pair over $\spec (k)$ and $p \in \cD$ such that $\cD$ is smooth at $p$. Then $\cX$ is smooth at $p$.
\end{Cor}

 Proposition \ref{Prop:no:different:on:stack} tells us that in replacing a surface pair $(X,D)$ with its associated AH pair $(\cX, \cD)$, the different on $D$ gets replaced by the stack structure on $\cD$. Our goal now is to make explicit the relation between the stack structure on $\cD$ and the different on $D$.
 
\begin{Lemma}\label{Lemma:AH:stack:iso:along:nodes}
Consider $(X,D)$ an lc surface and assume that $p \in D$ is a nodal point of $D$. Then $K_X+D$  is Cartier at $p$ and $X$ agrees with its AH stack in a neighbourhood of $p$.
\end{Lemma}

\begin{proof}
From Step 2 of the proof of \cite{Kollarsingmmp}*{Corollary 2.32}, it follows that the singularity of $X$ at $p$ is a cyclic quotient singularity. In particular, denoting the completed local ring $\widehat{\cO_{X,p}}$ by $R$, there
is a morphism $\pi:\spec(k[\![x,y]\!]) \to \spec(R)$ which is the quotient by a cyclic group $G$. The action of $G$ is étale in codimension one so $\pi^*K_{\spec(R)}=K_{\spec(k[\![x,y]\!])}$ and
the pair $(\spec(k[\![x,y]\!]),\pi^*D)$ is lc from \cite{Kollarsingmmp}*{Corollary 2.43}. But then $p^*D$ is nodal, and the group $G$ preserves the two branches, so up to an analytic change of coordinates we may suppose that $D=(xy)$ and $x,y$ are eigenvectors for $G$. Furthermore, the log canonical divisor is Cartier generated by $\frac{dx \wedge dy}{xy}$.
Then the action of $G$ on the log canonical divisor is trivial so the log canoical divisor descends to a Cartier divisor on $\spec(R)$.  
\end{proof}

We are left with understanding $\cD$ along the points where $D$ is smooth:
\begin{Lemma}\label{Lemma:relation:singularities:different}
 Let $(X,D)$ be an lc surface and $q \in D$ a closed point at which $D$ is smooth. Assume that $q$ appears in the different of $(X,D)$ with multiplicity $m$. Let $(\cX,\cD)$ be the
 twisted surface pair associated to $(X,D)$ and let $p \in \cD$ be the point lying over $q$ with stabilizer group $G:=\Aut_\cX(p)$. Then we have the following.
 \begin{enumerate}
  \item If $p$ is a smooth point of $\cD$, then $m=1-\frac{1}{|G|}$;
  \item If $p$ is a node, then $m=1$, $G=\bmu_2$ and $G$ acts by swapping the two branches of $\cD$.
 \end{enumerate}
 \end{Lemma}
 
 \begin{Remark} For case $(1)$, $m$ determines the stabilizer group since the stack is cyclotomic. \end{Remark} 
 
\begin{proof} Since the question is local, we may pass to an open subset and assume that $\cX$ has no stabilizers outside of $p$ and $D$ has no different outside of $q$.

For case $(1)$, consider the following commutative diagram.
 
 $$
 \xymatrix{\cD \ar[r]^\alpha \ar[d]_{\pi_D} & \cX \ar[d]^\pi \\
 D \ar[r]_a & X}
 $$
 
Letting $n$ be the index of $K_X + D$, we have the following equalities:
\begin{itemize}
\item $\pi^*(n(K_X + D)) = n(K_\cX + \cD)$ since $\pi$ is an isomorphism in codimension one,
\item $a^*(n(K_X + D)) = n(K_D + mq)$ by assumption on the different, and
\item $\alpha^*(n(K_\cX + \cD)) = nK_{\cD}$ by Proposition \ref{Prop:no:different:on:stack}.
\end{itemize}
Putting this together along with commutativity of the above diagram, we obtain 
$$
\pi_D^*(n(K_D+mq))= \pi_D^*a^*(n(K_X + D)) = \alpha^*\pi^*(n(K_X + D)) = nK_\cD,
$$
or simply $\pi_D^*(K_D + mq) = K_\cD$. Up to completing around $p$, we may assume that coarse map $\cD \to D$ is given by $\spec(k[\![x]\!]) \to \spec(k[\![x^r]\!])$
where $r=|G|$. In this case, one readily sees that $\pi_D^*(q) = rp$ and $\pi_D^*K_D = K_\cD -(r-1)p$. Putting this together with the previous equality, we have then $K_\cD = K_\cD-(r-1)p+rmp$ which gives $m=\frac{r-1}{r}$.

For point $(2)$, we may again replace $(\cX,\cD)$ with an étale neighboourhood of $p$ and assume the following:
 \begin{enumerate}
  \item there is an lc surface pair $(X',D')$ and a distinguished point $p' \in D'$;
  \item there is an action of $G$ on $X'$ preserving $D'$ and fixing only $q$, and
  \item $X=X'/G$, $D=D'/G$ and the quotient map sends $p'$ to $p$.
 \end{enumerate}
Since $D$ is smooth, we know that $G$ swaps the two branches of $D'$. Let $H$ be the normal subgroup of $G$ which preserves the branches of $D'$.
Then the pair $(X'/H,D'/H)$ is lc, the map $\pi:X' \to X'/H$ is étale away from $q$, and $D'/H$ is a nodal curve. But then from Lemma \ref{Lemma:AH:stack:iso:along:nodes}
the log canonical divisor $L$ of $(X'/H,D'/H)$ is Cartier. Since $\pi$ is étale in codimension one, $K_{X'}+D'\cong \pi^*L$. In particular, $H$ acts trivially on $(K_{X'}+D')_p$. Then $H=\{1\}$ since the action of $G$ on $(K_{X'}+D')_p$ is faithful (recall that $\cX \to B\bG_m$ is representable).
Therefore, any non-zero element of $G$ swaps the branches of $D'$ and $G \cong \bmu_2$. 
\end{proof}
\begin{Remark}
For other relations between the singularities of $X$ along $D$ and the different on $D$, one can consult \cite{Kollarsingmmp}*{Theorem 3.36}.
\end{Remark}
\begin{EG}
Consider $(\mathbb{A}^2_{u,v},D':=V(uv))$, with the action of $\mathbb{Z}/2\mathbb{Z}$ that sends $(u,v)\mapsto (-u,-v)$. Then $D'$ is invariant and the quotient is the surface pair $(X:=\spec(k[x,y,z]/xy-z^2),D:=V(z))$.
The divisor $D$ is still a nodal curve and the pair is lc since it is the quotient of an lc pair (see \cite{Kollarsingmmp}*{Corollary 2.43}).
Consider now the action of $G:=\bmu_2$ on $X$, which sends $(x,y,z) \mapsto (y,x,-z)$.
The action is free in codimension one and preserves $D$ so $(X/G,D/G)$ is again lc.

The log-canonical divisor on $X/G$ is not Cartier. Indeed, it suffices to show that the generator of $G$ acts nontrivially on a section of the log canonical divisor on $X$.
By adjunction, $\omega_X=(\omega_{\mathbb{A}^3}(X))\big|_{X}$ so a section of $\omega_X$ is given by $(\frac{dx \wedge dy \wedge dz}{xy-z^2})\big|_{X}$ which is invariant under $G$.
It follows that $G$ acts on a generator for $\omega_X(D)$ by $$\frac{dx \wedge dy \wedge dz}{xyz-z^3} \mapsto -\frac{dx \wedge dy \wedge dz}{zxy-z^3}$$
so the log canonical divisor is not Cartier on the quotient. Then the twisted pair associated to $(X/G,D/G)$ is $([X/G],[D/G])$, the divisor $D/G$ is smooth, but $[D/G]$ is not. By Lemma \ref{Lemma:relation:singularities:different}, we may conclude that the fixed point $p$ has different $1$ on $D/G$.

One can argue indirectly that $p$ must have different 1. Indeed the different is an effective divisor with coefficient at most 1. If it was not 1, then $(D/G,\Diff_{D/G}(0))$ would be klt and $(X/G,D/G)$ would be plt by inversion of adjunction. Then from \cite{Kollarsingmmp}*{Corollary 2.43}, the pairs $(X,D)$ and $(\mathbb{A}^2_{u,v},V(uv))$ would be plt but this is not the case.
\end{EG}

Finally we study how the normalization map $\cD^n \to \cD$ behaves around the node point $p \in \cD$ in the situation of point (2) of Lemma \ref{Lemma:relation:singularities:different}. 

\begin{Cor}\label{Cor:the:normalization:of:a:stacky:curve:when:we:swap:branches:is:a:scheme}
 Suppose that $(\cX,\cD)$ is an lc twisted surface pair and $p \in \cD$ is a node lying over a smooth point $q$ of $D$. Let $\cD^n$ be the normalization of $\cD$. Then the composition $\cD^n \to \cD \to D$ is an isomorphism in a neighbourhood of each $q$.
\end{Cor}

\begin{proof} It suffices to check the required map is an isomorphism after taking \'etale charts so without loss of generality, we may suppose that $\cD$ is isomorphic to
 $[\spec(k[x,y]/xy)/\bmu_2]$ with the action that sends $x \mapsto y$ and $y \mapsto x$. Since normalization commutes with étale
 base change, we have that the normalization $\cD^n$ is the quotient
 $$[\spec(k[x]) \sqcup \spec(k[y])/\bmu_2]$$ with the action $(x,y) \mapsto (y,x)$. Then $\cD^n$ is just isomorphic to the scheme $\spec( k[z])$ by the map $x,y \mapsto z$, and we can write the morphisms $\cD^n \to \cD \to D$ as 
  \begin{align*}\spec(k[x]) &\cong [\spec(k[x]) \sqcup \spec(k[y])/\bmu_2] \to \\ & \to [\spec(k[x,y]/xy)/\bmu_2] \to \spec(k[x+y]) \cong \spec(k[t]).\end{align*}
   The composition is the morphism induced by $t \to x+y \to x$ which is an isomorphism.\end{proof}

\section{Gluing morphisms for families of twisted stable surfaces}\label{Section:gluing}
In this section we produce
gluing morphisms which describe the boundary of the moduli of twisted stable surfaces in terms of moduli of twisted stable surface pairs (see Theorem \ref{thm:intro3}).
More precisely, in Subsection \ref{subsec:gluing:map}, we prove that one can functorially glue a family of twisted stable surface pairs over an arbitrary base to obtain a family of twisted stable surfaces.
In Subsection \ref{subsec:gluing:data}, we define an algebraic stack $\cG_{2,b}$ of gluing data. Finally, in Subsection
\ref{subsec:boundary:strata} we show that the image of $\cG_{2,v} \to \cK_{2,v}$ stratifies $\cK_{2,v}$ into finitely many boundary components.

\subsection{Gluing morphisms for twisted stable surfaces}\label{subsec:gluing:map} We begin this subsection by recalling some of the results
in \cite{Kollarsingmmp}*{Chapter 5}.

Consider an slc stable surface $X'$ and let $X$ be its normalization with conductor $D \subseteq X$. Since $X'$ is nodal in codimension 1, the normalization $X \to X'$ induces a rational and generically fixed point free involution $D \dashrightarrow D$. By \cite{Kollarsingmmp}*{Proposition 5.12}, this gives a generically fixed point free involution on the normalization $\tau:D^n \to D^n$ which preserves the different. It follows from \cite{Kollarsingmmp}*{Theorem 5.13} that the converse also holds. Namely, the data of a stable lc pair $(X,D)$ and a generically fixed point free involution $D^n \to D^n$ which preserves the different determine a unique stable surface $X'$. The goal of this subsection is to understand \cite{Kollarsingmmp}*{Theorem 5.13} in terms of twisted stable surfaces, which will give us a generalization to families over arbitrary bases.

We begin with two auxiliary results:

\begin{Lemma}\label{Lemma:we:have:tau:on:the:stack}
Let $(\cX,\cD)$ be a normal AH twisted stable surface pair, and let $\cD^n$ be the normalization of $\cD$. Denote the coarse space of $\cD^n$ by $D^n$ and let $\tau:D^n \to D^n$ be a different preserving involution on $D^n$. Then there is a unique involution
$\sigma:\cD^n \to \cD^n$ lifting $\tau$:
$$\xymatrix{\cD^n \ar[r]^\sigma \ar[d]_{\text{coarse space}} & \cD^n \ar[d]^{\text{coarse space}}\\
D^n \ar[r]_\tau & D^n}.$$
\end{Lemma}

\begin{proof}
It follows from Corollary \ref{Cor:the:normalization:of:a:stacky:curve:when:we:swap:branches:is:a:scheme} and Lemma \ref{Lemma:relation:singularities:different}
that the coarse space morphism $\cD^n \to D^n$ is an isomorphism away from the points where the different has a coefficient $0 < c < 1$.
These correspond to smooth points of $\cD$ with nontrivial stabilizer and it suffices check that we can extend the morphism along such points. 

From Lemma \ref{Lemma:relation:singularities:different}, two points with stabilizers of different orders cannot be swapped.
Locally around such points on the coarse space, we have an involution $\tau$ of $Y:= \spec(k [\![s]\!])$ that we wish to extend uniquely to an
involution of the $m^{th}$ root stack $\cY:=[\spec(k[\![t]\!])/(t \mapsto \zeta_m t)]$ for some $m$.
Here $\zeta_m$ is a primitive $m^{th}$ root of unity and $t^m = s$.
The required assertion follows from the universal property of root stacks (see \cite{Ols16}*{10.3}). Such a root stack is the universal objects for triples $(Y,L,v)$ where $Y$ has a morphism $f:Y \to \spec(k[\![s]\!])$, $L$ is a line bundle on $Y$, and $v$ is a section of $L$. Moreover, we require that there is an isomorphism $L^{\otimes m}  \to f^*\cO_{\spec(k[\![s]\!])} $ sending $v^{\otimes m} \mapsto s$.
\end{proof}
\begin{Notation}Let $\nu:\cD^n \to \cD$ be the normalization morphism. We will denote the nodes of $\cD$ by $\cN$ and set $\Delta:=\nu^{-1}(\cN)$. Then $\Delta$ is a scheme from Lemma \ref{Lemma:AH:stack:iso:along:nodes} and Corollary \ref{Cor:the:normalization:of:a:stacky:curve:when:we:swap:branches:is:a:scheme}.
\end{Notation}

\begin{Lemma}\label{Lemma:we:have:the:map:on:D:divided:by:tau:over:the:pt}
Consider a normal twisted stable surface pair $(\cX,\cD)$. Let $\cD^n$ be the normalization of $\cD$ and let $\tau:\cD^n \to \cD^n$ be an involution on $\cD^n$ which is generically fixed point free and preserves $\Delta$. Let $X$ and $D$ be the coarse spaces of $\cX$ and $\cD$ respectively. Consider the stable surface $X'$ obtained from $X$, $D$ and $\tau$ using \cite{Kollarsingmmp}*{Theorem 5.13}. Let $\cX'$ be the AH stack
associated to $(X',\omega_{X'})$. Then $\cX$ is the normalization of $\cX'$ and there is a map $h:[\cD^n/\tau] \to \cX'$ which makes the following diagram commute.\footnote{For properties of group quotients of DM stack, see \cite{Rom05}.}
$$ \xymatrix{\cD^n \ar[r] \ar[d] & \cX \ar[d] \\
[\cD^n/\tau] \ar[r]^-h & \cX'}$$
\end{Lemma}

\begin{proof} Recall that $\cX'$ is a scheme away from a finite set of points. Therefore the normalization $\nu:\cY \to \cX'$ has $X$ as coarse space, and if we denote by $p:\cY \to X$ the coarse space map, $p_*(\nu^* \omega_{\cX'}^{\otimes n}) \cong (\omega_{X}(D))^{[n]}$ for every $n$ by \cite{Kollarsingmmp}*{5.7.1}. But $\omega_{\cX'}$ is a uniformizing line bundle for $\cX'$, and $\nu$ is representable, so the morphism $\cY \to B\bG_m$ induced by $\nu^*(\omega_{\cX'})$ is also representable. It follows from Lemma \ref{Lemma:iso:to:AH:stack:if:P(F):scheme} that $\cY \cong \cX$. In particular we have a map $\cX \to \cX'$ which induces $f:\cD^n \to \cX'$. There is an open dense subset $U$ of $\cD^n$ where $\cD^n$ is a scheme and the two maps $f\big|_{U}, (f \circ \tau)\big|_{U} :U \to \cX'$ agree. By \cite{DH18}*{Lemma 7.3}, the arrows $f$ and $f \circ \tau$ agree and so they induce a map $h:[\cD^n /\tau] \to \cY$ making the diagram commute.
\end{proof}

The canonical line bundle $\omega_{\cX'}$ induces a map $\cX' \to B\bG_m$ and we obtain a map $[\cD^n/\tau] \to B\bG_m$ given by composition with the map $h$ from Lemma \ref{Lemma:we:have:the:map:on:D:divided:by:tau:over:the:pt}. We wish to understand this map $[\cD^n/\tau] \to B\bG_m$ in terms of $\cX$, $\cD^n$ and $\tau$. First, observe that $\tau$ induces a map $\psi:[\cD^n/\tau] \to B\bmu_2$. We can identify $\Pic(B\bmu_2)$ with the characters of $\bmu_2$ and we denote by $\chi_{B\bmu_2}$ the line bundle corresponding to the sign representation. Given a morphism $\cY \to B\bmu_2$, we will denote by $\chi_{\cY}$ the pull back of $\chi_{B\bmu_2}$ to $\cY$. 

Now, the involution $\tau$ naturally acts on
$\omega_{\cD^n}(\Delta)$ since $\tau$ preserves the different $\Delta$. Thus $\omega_{\cD^n}(\Delta)$ descends to a line bundle $\cG$ on $[\cD^n/\tau]$.

\begin{Prop}\label{Prop:line:bundle:which:gives:rep:morphism:Ddividedbytau:to:BGm} In the situation above, the line bundle corresponding to $[\cD^n/\tau] \to B\bG_m$ is $\cG \otimes \chi_{[\cD^n/\tau]}$.
\end{Prop}

\begin{proof} Let $\pi:\cD^n \to [\cD^n/\tau]$ be the quotient map, let $\cH:=h^*\omega_{\cX'}$, and let $\sigma:\bmu_2 \times \cD^n  \to \cD^n$ be the action induced by $\tau$. The key ingridient is \cite{Kollarsingmmp}*{Proposition 5.8}.

From the diagram of Lemma \ref{Lemma:we:have:the:map:on:D:divided:by:tau:over:the:pt} and Proposition \ref{Prop:no:different:on:stack}, we have $\pi^*\cH \cong \omega_{\cD^n}(\Delta)$. The line bundle $\cH$ is uniquely determined by the isomorphism $\alpha:\sigma^*( \omega_{\cD^n}(\Delta)) \to p_2^*(\omega_{\cD^n}(\Delta))$ (see \cite{Rom05}*{Example 4.3}). To determine $\alpha$, it suffices to restrict $\cH$ to a dense open subset of $\cD^n$: we can assume that $\cX$ and $[\cD^n/\tau]$ are schemes. Now from \cite{Kollarsingmmp}*{Proposition 5.8}, the sections of $\cH$ and those of $\cG \otimes \chi_{B\bmu_2}$ agree so $\cH \cong \cG \otimes \chi_{B\bmu_2}$.
\end{proof}

\begin{Notation}
 We will denote by $\cC$ the relative coarse moduli space of the map $[\cD^n/\tau] \to B\bG_m$ induced by $\cG \otimes \chi_{[\cD^n/\tau]}$.
\end{Notation}

Now, since $\tau$ sends $\Delta$ to itself, we can factor $\Delta \to \cC$ as $\Delta \to [\Delta/\tau] \to [\cD^n/\tau] \to \cC$ and we have the following commutative diagram. 
$$\xymatrix{\Delta \ar[dd] \ar[dr]\ar[rr] & & \cN \ar[dr] & & & \\
& \cD^n \ar[dd] \ar[rr]^\nu & & \cD \ar[r]  &\cX  & \\
[\Delta/\tau]\ar[dr]^-f & & & & & \\
& \cC & &  &  & }$$

\begin{Lemma}\label{Lemma:right:stacky:structure:on:quotient:of:delta:to:have:closed:embedding}
 The map $f:[\Delta/\tau] \to \cC$ is a closed embedding. In particular, the composition $[\Delta/\tau] \to \cC \to B\bG_m$ is representable.
\end{Lemma}
\begin{proof}

 Since $\Delta \to \cD^n$ is a closed embedding, then $[\Delta/\tau] \to [\cD^n/\tau]$ is also a closed embedding. Therefore the claim follows if we can prove that $\psi:[\cD^n/\tau] \to \cC$ is an isomorphism along $[\Delta/\tau]$. Let $\pi:\cD^n \to [\cD^n/\tau]$ be the quotient map.
 
 If $q \in \Delta$ is such that $\tau(q) \neq q$, then $\pi(q)$ has no stabilizer so $\psi$ is necessarily an isomorphism in a neighboourhood of $\pi(q)$. Assume then that $\tau(q)=q$. Formally locally around $q$, we can replace $\cD^n$ with $\spec(k[\![z]\!])$ and assume the involution $\tau$ sends $z \mapsto -z$ where $z$ is a uniformizer for $q$. Since $dz/z$ is a generator of $\omega_{\cD^n}(\Delta)$, $\tau$ acts trivially on the fiber of $\omega_{\cD^n}(\Delta)$ at $q$. Therefore, $\tau$ acts
 faithfully on the fiber of $\cG \otimes \chi_{[\cD^n/\tau]}$ at $q$ and so by \cite{AH}*{Proposition 2.3.10}, $\psi$ is an isomorphism at $\pi(q)$.
\end{proof}

We now define a DM stack $\cS$ with two representable morphisms $g_1:\cN \to \cS$ and $g_2:[\Delta/\tau] \to \cS$, such that the two morphisms $\Delta \to [\Delta/\tau] \to \cS$ and
$\Delta \to \cN \to \cS$ are isomorphic. We will denote with roman letters the coarse space of a DM stack, as in the conventions.
We define $S$ to be the pushout in algebraic spaces of the diagram below:
$$\xymatrix{\Delta \ar[d] \ar[r] & N \ar[d]\\\Delta/\tau \ar[r] & S.}$$
Thus $S$ is a disjoint union of points. We determine a stack $\cS$ with coarse space $S$ by requiring that for every point $q \in S$, the group
$\Aut_q(S)$ is $\bmu_2$ if there is a $\bmu_2$ stabilizer on a point of $g_1^{-1}(q) \cup g_2^{-1}(q)$ and trivial otherwise. Recall these are the only options for the stabilizers of $[\Delta/\tau]$ and $\cN$. Now by Lemma \ref{Lemma:unique:representable:arrow:gerbes:over:pt}, the morphisms $g_i$ are uniquely determined.
\begin{Lemma}\label{Lemma:unique:representable:arrow:gerbes:over:pt}
 Let $a,b$ be positive integers with $a$ even. Then, up to isomorphism, there are unique representable
morphisms $\spec(k) \to B\bmu_b$ and $B\bmu_2 \to B\bmu_a$.
\end{Lemma}
\begin{proof}Since $k$ is algebraically closed, there is a unique $\bmu_b$-torsor over $\spec(k)$, so there is a unique morphism $\spec(k) \to B\bmu_b$.
For the claim about $B\bmu_2$, consider the presentation of $B\bmu_2$ given by $\bmu_2 \rightrightarrows \spec(k)$, where one arrow $\sigma$ is the action of $\bmu_2$ on the point and the other one $\pi_2$ is the structure map. Note however, since $\mu_2$ acts trivially on $\spec(k)$ these are both just the structure map. Then a morphism $B\bmu_2 \to B\bmu_a$ is equivalent to the data of:
\begin{enumerate}
    \item A morphism $f:\spec(k) \to B\bmu_a$, and
    \item An isomorphism $\psi $ in $B\bmu_a(\bmu_2) \cong B\bmu_a(\spec (k)) \times B\bmu_a(\spec (k))$ between $f \circ \sigma $ and $f \circ \pi_2$.
\end{enumerate}
Moreover, $\psi$ must satisfy the cocycle condition. More precisely, let $m:\bmu_2 \times \bmu_2 \to \bmu_2$ be the multiplication and $p_{2,3}$ the projection onto the second factor $\bmu_2 \times \bmu_2 \to \bmu_2$. Then we require that $p_{2,3}^*\psi \circ (\Id_{\bmu_2}\times \sigma)^* \psi= m^* \psi$.

Now, there is a unique morphism $\spec(k) \to B\bmu_a$ from the previous point. An isomorphism $\psi$ as above is the data of two automorphisms $(\alpha,\beta)$ in $B\bmu_a(\spec(k))$.
We will denote by $\{\pm 1\}$ the two points of $\bmu_2$ and by $(1,1),(1,-1),(-1,1),(-1,-1)$ the four points of $\bmu_2 \times \bmu_2$. To fix the notation, $\alpha$ will be the automorphism over $1$ and $\beta$ the one over $-1$. Then the arrows $p_{2,3}$, $\Id_{\bmu_2} \times \sigma$ and $m$ behave as follows:
\begin{itemize}
\item $p_{2,3}(1,1)=p_{2,3}(-1,1)=1$ and $p_{2,3}(1,-1)=p_{2,3}(-1,-1)=-1$;
\item $\Id_{\bmu_2} \times \sigma(1,1)=\Id_{\bmu_2} \times \sigma(1,-1)=1$ and $\Id_{\bmu_2} \times \sigma(-1,1)=\Id_{\bmu_2} \times \sigma(-1,-1)=-1$;
\item $m (1,1)=m (-1,-1)=1$, whereas $m (1,-1)=m(-1,1)=-1$.
\end{itemize}
In particular, the cocycle condition is the following equality of automorphisms over $\mu_2 \times \mu_2$:
$$(\alpha, \beta,\alpha,\beta) \circ (\alpha,\alpha,\beta,\beta)=(\alpha, \beta, \beta,\alpha).$$
Therefore we must have $\alpha \circ \beta = \beta$ which implies $\alpha=\Id$ and that $\beta \circ \beta =\Id$. Thus the only $\psi$ which satisfy the cocycle condition are $(\Id,\Id)$ and $(\Id, \beta)$ where $\beta$ is the unique element of $\bmu_a$ with $\beta^2=\Id$ and $\beta \neq \Id$.
This means that there are exactly two morphisms $B\bmu_2 \to B\bmu_a$. The first one is the composition $B\bmu_2 \to \spec(k) \to B\bmu_a$ and the second one is induced by the morphism $\bmu_2 \hookrightarrow \bmu_a$ as in \cite{Ols16}*{Exercise 10.F}. Only the second is representable.
\end{proof}
\begin{Prop}\label{Prop:we:have:the:map:P:to:AH:Stack} With notation as above, let $\cX'$ be the AH stack associated to $(X' \to \spec(k),\omega_{X'})$. Then there is a unique representable morphism $\cS \to \cX'$ which makes the diagram below commutative.
$$\xymatrix{\Delta \ar[dd] \ar[dr]\ar[rr] & & \cN \ar[dr] \ar[dd]|!{[d]}\hole & &  \\
& \cD^n \ar[dd] \ar[rr] & & \cD \ar[r]  &\cX \ar[dd]    & \\
[\Delta/\tau]\ar[dr]^-f \ar[rr]|!{[r]}\hole & & \cS \ar@{.>}[rrd] & & \\
& \cC \ar[rrr] & &  & \cX' }$$
\end{Prop}
\begin{proof} It suffices to consider the following subdiagram:
$$\xymatrix{& \cN \ar[d] \ar[dr] &\\
[\Delta/\tau] \ar[r] \ar@/_1pc/[rr]&\cS \ar@{.>}[r] & \cX'.}$$

Up to considering one connected component of $\cS$ at a time, we can assume that $S \cong \spec(k)$.
Now, recall that
the diagram on coarse spaces is a pushout and the diagram of solid arrows commutes. Then
if we replace all the stacks above with their
coarse spaces, we have an arrow $q:S=\spec(k) \to X'$. In particular, if we denote $\bmu_a:=\Aut_{\cX'}(q)$, we have a closed embedding
$B\bmu_a \to \cX'$ (this follows from \cite{Ols16}*{Theorem 11.3.1}). The arrows $[\Delta/\tau] \to \cX'$ and $\cN \to \cX'$
factor through $B\bmu_a \to \cX'$, so in the diagram above we can replace $\cX'$ with $B\bmu_a$. Now the claim follows from the description of $\cS$ and Lemma \ref{Lemma:unique:representable:arrow:gerbes:over:pt}.
\end{proof}

We recall a theorem due to David Rydh on the existence of pinchings in algebraic stacks:

\begin{Teo}[\cite{Rydh}*{Theorem A.4}]\label{Teo:Rydh}
Consider a diagram of solid arrows between algebraic stacks over $B$ as below, with $i$ a closed embedding and $f$ a finite morphism.
$$\xymatrix{\cX \ar@{^{(}->}[r]^i \ar[d]_f& \cY \ar@{.>}[d]^{f'} \\ \cX' \ar@{.>}[r]_{i'} & \cY'}$$
Then there are arrows $i'$ and $f'$ over $S$, such that the resulting diagram is a pushout in algebraic stacks. Moreover, $i'$ is a closed embedding and
$f'$ is integral and an isomorphism away from $\cX$.
If we take the topological spaces of the algebraic stacks above, the corresponding square is a pushout in topological spaces.
Finally, if $\cX$, $\cY$ and $\cX'$ are proper over $B$, then also $\cY'$ is proper over $B$ (recall that for us all the algebraic
stacks are of finite type).
\end{Teo}

\begin{Remark}\label{Remark:the:construction:of:a:pushout:is:etale:local}The construction of $\cY'$ can be performed smooth locally.
In particular, if we take $\spec(A') \to \cY'$ a smooth morphism,
and we pullback $i'$, $f'$, $i$ and $f$ through it, the corresponding diagram will be a pushout in schemes.\end{Remark}

The following Lemma is the last technical result we need before proving Theorem \ref{Teo:gluing:morphisms}.
\begin{Lemma}\label{Lemma:the:normalizatio-conductor:diagram:with:D:is:a:pushout}
 The following diagram is a pushout in algebraic stacks:
 $$\xymatrix{\Delta \ar[r] \ar[d]&\cN\ar[d] \\ \cD^n \ar[r] & \cD}$$
\end{Lemma}
\begin{proof} From \cite{Rydh}*{Theorem A.4}, the pushout $\cD'$ of the diagram without $\cD$ exists; we need to show that $\cD' \cong \cD$. Since $\cD'$ can be constructed étale locally, we can assume that $\cD$, $\cD^n$, $\cN$ and $\Delta$ are schemes. Then the conclusion follows from \cite{ks}.
\end{proof}
Now we aim at constructing $\cX'$ from the data of $\cD$, $\tau$ and $\cX$. The situation is sumarized in the following diagram, where the dashed arrows are pushouts, the wavy ones come from the universal property of a pushout, and all the arrows are representable. From the definition of $\cS$, the morphism
$[\Delta/\tau] \to \cS$ is finite. Therefore from \cite{Rydh}*{Theorem A.4} the two maps $[\Delta/\tau] \to \cS$ and $[\Delta/\tau] \to \cC$ have a
pushout $\cC \to \cQ$ and $\cS \to \cQ$. Then we have a morphism $\alpha: \cQ \to \cX'$. To check that $\alpha$ is representable, it suffices to
check that it is injective on automorphisms of closed points. But from the description of $\cQ$ given in \cite{Rydh}*{Theorem A.4}, it suffices that $\cC \to \cX'$ and $\cS \to \cX'$ are representable, which is true.

Then by Lemma \ref{Lemma:the:normalizatio-conductor:diagram:with:D:is:a:pushout}, we have a morphism $\cD \to \cQ$ which
is representable by \cite{Rydh}*{Theorem A.4}. One can check that it is also quasifinite, and $\cD$ is proper, so $\cD \to \cQ$ is finite.
Finally, again from \cite{Rydh}*{Theorem A.4}, the two maps $\cD \to \cX$ and $\cD \to \cQ$ admit a pushout $\cZ$.
As before, we have a representable morphism $\Phi:\cZ \to \cX'$.
\begin{equation}\label{eqn:big:diagram}
\xymatrix{\Delta \ar[dd] \ar[dr]\ar[rr] & & \cN \ar[dr] \ar[dd]|!{[d]}\hole & & & \\
& \cD^n \ar[dd] \ar[rr] & & \cD \ar[r] \ar@{~>}[dd]  &\cX \ar@{-->}[dd] \ar[dr]   & \\
[\Delta/\tau] \ar[dr]^-f \ar[rr]|!{[r]}\hole & & \cS \ar[rrr]|!{[r]}\hole|!{[rr]}\hole \ar@{-->}[dr] & & & \cX'\\
& \cC \ar@{-->}[rr]& & \cQ \ar@{-->}[r] & \cZ.\ar@{~>}[ur]_-\Phi  & }
\end{equation}

\begin{Prop}\label{Prop:pushout:gives:the:AH:stack:over:speck} The morphism $\Phi$ is an isomorphism. In particular, $\cZ$ is Gorenstein.
\end{Prop}

\begin{proof} From the univesal property of the pushouts and coarse spaces, one can check that taking the coarse space commutes with pushouts. In particular, by \cite{Kollarsingmmp}*{Chapter 5}, the morphism $\Phi$ is an isomorphism on coarse spaces.
Furthermore, $\cZ$ is seminormal since it is the pushout of seminormal stacks. 

Observe that $\Phi$ is an isomorphism away from a finite set of points $\{q_1,...,q_r\}$. Let $p:=q_i$ and let $m:=|\Aut_{\cX'}(p)|$ be the index of $K_{X'}$ at $p$. If $n:=|\Aut_{\cZ}(p)|$, the line bundle $\Phi^*(\omega_{\cX'})^{\otimes n}$ descends to a line bundle in a neighborhood of $p \in X'$ and agrees with $\omega_{X'}^{[n]}$ in codimension one so the index $m$ divides $n$. On the other hand, by representability, $n$ divides $m$. Thus $m = n$ and the morphism $\Phi_p: \Aut_{\cZ}(p) \to \Aut_{\cX'}(p)$ is bijective. We conclude that $\Phi$ is a proper morphism of seminormal DM stacks such that $\cZ(\spec(k)) \to \cX'(\spec(k))$ is an equivalence so $\Phi$ is an isomorphism.
\end{proof}

\begin{EG}
Consider $(X,D)=(\mathbb{A}^2_{x,y},V(xy))$, so $D$ is a nodal curve and consider the involution $\tau:D \to D$ which sends $(x,y) \mapsto (-x,-y)$. Call the resulting slc surface $X'$.

Observe that $(\omega_X(D))\big|_{D}=\omega_D$ is generated by the section $\frac{dx \wedge dy}{xy}$ which is fixed by $\tau$. Thus the action on $\omega_D \otimes \chi$ is non-trivial and $\cC \cong [D^n/\tau]$. Then $\cN = p$ is a node, $\Delta = q_1 \sqcup q_2$ lying above the node which are both fixed points of $\tau$, and $\cS = [p/\bmu_2] = B\bmu_2$. It follows that $\cQ = [D/\tau]$ and the AH stack $\cX'$ of $X'$ is the pushout in stacks of $D \hookrightarrow \mathbb{A}^2$ and $D \to [D/\tau]$. In particular it has a $\bmu_2$ stabilizer at the origin. Therefore $\omega_{X'}$ has index $2$ at the origin. 

Using a computer algebra system, one can compute equations for $X'$ directly from the presentation above. They are given by the following ideal in $k[a,b,c,d,e]$:
$$
(cd - ae, bd - ce, ab - c^2, de - c^3, bc^2 - e^2, ac^2 - d^2).
$$
\end{EG}

We are now ready to prove the main theorem of this subsection, which states that the construction of
$\cZ$ above can be carried out in families over an arbitrary base.

\begin{Teo}\label{Teo:gluing:morphisms}Consider a twisted stable family of normal surface pairs $(f:\sX\to B,\phi: \omega_f \to \sL)$.
Assume that $\sD_\phi \to B$ has a simultaneous normalization,
namely that there is a flat proper morphism $\sD^n \to B$ and a morphism $\nu:\sD^n \to \sD_\phi$ such
that for every $b \in B$, the map $\nu_b:\sD^n_b \to (\sD_\phi)_b$ is a normalization. Let $\{\sigma_i:B \to \sD^n\}$ be
disjoint sections which surject onto the locus where
$\sD^n \to \sD_\phi$ is not an isomorphism.
Finally let $\tau:\sD^n \to \sD^n$ be an involution over $B$ which preserves the closed substack
$\bigsqcup \sigma_i(B)$ and which is fiberwise generically fixed point free. 

Then there exists a DM stack $\sZ$ which fits in the commutative diagram below.
$$ \xymatrix{\sD^n \ar[r] \ar[d] & \sX \ar[d] \\ [\sD^n/\tau] \ar[r] & \sZ}$$
Moreover, $\sZ \to B$ is flat and proper, and for every $p\in B$,
the fiber $\sZ_p$ is the AH stack of the stable surface obtained from gluing data $D_p^n$, $\tau\big|_{D_p^n}$ and $X_p$ by \cite{Kollarsingmmp}*{Theorem 5.13}. 
\end{Teo}
\begin{proof}Throughout the proof, we will drop the subscript $\phi$ on $\cD_\phi$. We will denote by $\Delta$ the closed substack, which is actually a scheme, given by the sections
$\sigma_i$. Furthermore, we can assume that $S$ is connected.

From Proposition \ref{Prop:no:different:on:stack} we have that $\sL\big|_{\sD^n}\cong \omega_{\sD^n/B}(\Delta)$. It descends to a line bundle
$\sG$ on $[\sD^n/\tau]$, and we let $[\sD^n/\tau] \to \sC$ be the relative coarse moduli space of the map $[\sD^n/\tau] \to B\bG_{m,B}$ induced by
$\sG \otimes \chi_{[\sD^n/\tau]}$. The induced morphism $\sC \to B$ is flat and the construction of $\sC$ commutes with arbitrary base change $B' \to B$ by \cite{AOV}*{Proposition 3.4}. 

Consider the quotient stack $[\Delta/\tau]$. Now, $\Delta \subseteq \sD^n$ is a disjoint union of sections of $\sD^n \to B$, so $\Delta \cong \bigsqcup B$. Since $\tau$ acts fiberwise and preserves $\Delta$, and since $B$ is connected, we have
that $[\Delta/\tau] \cong [\Delta_b/\tau]
\times B$ for any $b \in B$. In other words, $[\Delta/\tau]$ is a constant family over $B$.
Similarly, if we denote by $\sN$ the singular locus of $\sD \to B$, then $\sN \cong \sN_b \times B$ and $\Delta\cong \Delta_b \times B$.
For any fixed $b$, considering $\Delta_b, [\Delta_b/\tau ]$ and $\sN_b$, we can construct $\cS_b$ as in Proposition \ref{Prop:we:have:the:map:P:to:AH:Stack}. Now we can define $\cS:=\cS_b \times B$. Then $\cS$ fits into a commutative diagram as below and by construction its formation commutes with base change.
$$\xymatrix{\Delta \ar[r] \ar[d] & \sN \ar[d] \\ [\Delta/\tau] \ar[r]& \sS}$$

Now we proceed by constructing the analogous pushouts from Diagram \ref{eqn:big:diagram} and the discussion before it. Along the way, we will need to check that our pushouts, which are pinchings as in Theorem \ref{Teo:Rydh}, commute with arbitrary basechange $B' \to B$. This will be checked étale locally (see Remark \ref{Remark:the:construction:of:a:pushout:is:etale:local}) using Lemma \ref{Lemma:pushout:commutes:with:bc}.

First we have the following diagram which, as in Lemma \ref{Lemma:the:normalizatio-conductor:diagram:with:D:is:a:pushout}, we claim that is a pushout:
$$\xymatrix{\Delta \ar[d] \ar[r] & \sN \ar[d] \\ \sD^n \ar[r] & \sD.}$$
Indeed, $\Delta, \sN, \sD^n$ and $\sD$ are flat over $B$, so we can check that the diagram above is a pushout after pulling back along $\spec(k)
\to B$, but this is the content of Lemma \ref{Lemma:the:normalizatio-conductor:diagram:with:D:is:a:pushout}.

Next we construct the analagous pushouts to the dashed ones in Diagram \ref{eqn:big:diagram}. This produces the desired proper morphism $\sZ \to B$ which is flat and commutes with base change by the above discussion. So $\sZ$ satisfies the claimed properties.
\end{proof}
In the proof of Theorem \ref{Teo:gluing:morphisms}, we needed the following technical result to check that the gluing construction commutes with base change and produces a flat family. 
\begin{Lemma}\label{Lemma:pushout:commutes:with:bc} Let $R$ be a ring, and consider two homomorphisms of
$R$-algebras $f:A' \to B'$ and $g:B \to B'$ with $A := A' \times_{B'} B$ their fiber product.
Assume that $B'$ is flat over $R$ and that $(g,-f):B\times A' \to B'$ is surjective. Then the square
$$
\xymatrix{A' \otimes_R S \ar[r] & B' \otimes_R S \\ A \otimes_R S \ar[r] \ar[u] & B \otimes_R S \ar[u]}
$$
is cartesian for any ring homomorphism $R \to S$. Moreover, if $A'$ and $B$ are flat over $R$, so is $A$.\end{Lemma}
\begin{proof} First note that the following sequence is exact:
$$ 0 \to A \to B \times A' \xrightarrow{(g,-f)} B' \to 0.$$
Since $B'$ is flat over $R$, we also have that 
$$ 0 \to A\otimes_R S \to (B \times A')\otimes_R S \to B'\otimes_R S \to 0$$
is exact. On the other hand, $(B \times A')\otimes_R S \cong (B \otimes_R S) \times (A'\otimes_R S)$, so $A \otimes_R S$ is the pullback of $B \otimes_R S \to B' \otimes_R S$ and $A' \otimes_R S \to B' \otimes_R S$, proving the first claim. The second claim follows by applying the long exact sequence of $\operatorname{Tor}_i$ to the first exact sequence in the proof above.
\end{proof}
\subsection{Gluing data}\label{subsec:gluing:data}
In this subsection we package the information of a gluing data into an algebraic stack $\cG_{2,v}$. Therefore, Theorem \ref{Teo:gluing:morphisms}
produces a gluing morphism $\cG_{2,v} \to \cK_{2,n}^\omega$ which on the level of points agrees with the gluing morphism
of \cite{Kollarsingmmp}*{Theorem 5.13}.
\begin{Prop}\label{Prop:we:have:the:stack:curly:G}
There is an algebraic stack $\cG_{2,v}$ which parametrizes the following objects. Over a scheme $B$, the objects of $\cG_{2,v}(B)$ are quadruples $$((f:\cX \to B, \phi:\omega_f \to \cL); g:\cD^n \to \cD_\phi; \{\sigma_i:B \to \cD^n\}; \tau:\cD^n \to \cD^n) $$
where:
\begin{enumerate}
    \item $(f:\cX \to B, \phi:\omega_f \to \cL)$ is an object of $\cK_{2,v}(B)$;
    \item  $\cD^n \to B$ is a flat family of orbifold smooth curves and $g$ is a simultaneous normalization;
    \item For a certain $n$, we have $n$ disjoint sections $\sigma_i$ of $\cD^n \to B$, such that
    $g \circ \sigma_i (b) \in (\cD^{\text{sing}}_\phi)_b$ and such that $g$ is an isomorphism away from $\bigsqcup \sigma_i(B)$, and
    \item $\tau$ is a generically fixed point free involution which preserves $\bigsqcup \sigma_i(B)$.
\end{enumerate}
The morphisms are pullback diagrams which satisfy the obvious commutativity conditions.
\end{Prop}
\begin{proof}
We will construct $\cG_{2,v}$ one condition at the time.
It suffices to construct $\cG_{2,v}$ for a fixed choice of $n$ of point (3), and then to construct $\cG_{2,v}$ by taking an union over $n \in \mathbb{N}$. Therefore, from now on we consider the number of sections $n$ as part of the data.

First,
consider $\sD \to \cK_{2,v}$ the universal divisor.
Consider the stack $\sO:=\cO rb^{\cL}$ parametrizing polarized orbifold curves, and let $\sC \to \sO$ be the universal curve. Consider now $\sH_1:=\cH om_{\sO \times \cK_{2,v}}(\sC \times\cK_{2,v}, \sO \times \sD) $, where for the definition of this Hom stack we refer to \cite{AOV}*{Appendix C}.
Over $\cH_1$ we have an universal curve $\sC_1$, obtained from the morphism $\sH_1 \to \sO$. Consider then $\sH_2:=\cH om_{\sH_1}(\sH_1,\sC_1)^{\times_{\sH_1}^n}$, which is the stack that parametrizes $n$ sections of $\cC_1 \to \cH_1$. Let $\sC_2$ be the universal curve of $\sH_2$,
and finally consider $\sH_3:=\cH om_{\sH_2}(\sC_2,\sC_2)$. Over a base $B$, the objects of $\sH_3(B)$ are the following:
\begin{enumerate}
\item $(f:\cX \to B, \phi:\omega_f \to \cL)$, an object of $\cK_{2,v}(B)$;
    \item  $\cC \to B$ is a flat family of orbifold nodal curves and a map $g:\cC \to \cD$;
    \item  $n$ sections $\sigma_i:B \to \cC$, and
    \item $\tau:\cC \to \cC$ a morphism.\end{enumerate} Let $\sC$ be the universal curve over $\sH_3$, and let $\sD_3 \to \sH_3$ be the pulll-back of $\sD$.
    
    Now, according to \cite{AOV}*{Appendix C}, there is an open substack of $\sH_1$ which parametrizes representable morphisms $g$.
    So up to replacing $\sH_1$ with this open substack, we can assume $g$ to be representable. We will denote with $\pi:\sC \to \sC^{\text{r.c.}}$ the relative coarse space of $\sC \to \sH_3$, and let $U \subseteq \sC$ to be the locus where $\pi$ is an isomorphism, and $\sC \to \sH_3$ is smooth. Having the sections $\sigma_i$ to be disjoint, and to map to $U$, is an open condition.
    Then up to shriking $\sH_3$, we can assume $\sigma_i$ to be disjoint, and to map to $U$. 
    
    From the upper semicontinuity of the dimension of the fibers, we can also assume that the morphism $\sC \to \sD$ is quasi-finite. But then it is finite since it is also proper and representable. Consider now the morphism $\cO_{\sD} \to g_* \cO_{\sC}$, let $\sK_1$ be its kernel, and let $\sS_1$ be the support of $\sK_1$. Then $\sS_1 \to \sH_3$ is proper. Thus from the upper-semicontinuity of the dimension of the fiber, applied to $\sS_1 \to \sH_3$, there is an open substack where $\sS_1$ is empty so $g$ is dominant. Since $g$ is proper it must also be surjective on this subset. Next let $\sK_2$ be the cokernel of $\cO_{\sD} \to g_*\cO_{\sC}$ and $\sS_2$ its support. As before, by upper-semicontinuity there is an open substack where $\sS_2$ is $0$-dimensional and so $g$ is generically an isomorphism. Since it is also a representable morphism of nodal curves, the locus where it is not an isomorphism must be contained in the nodes. 
    
    To recap, we have now cut out an algebraic stack where $g$ is a simultaneous normalization and $\sigma_i(B)$ are disjoint sections. We need to identify the locus where:
    \begin{itemize}
        \item $(g \circ \sigma_i )(b) \in \sD_b^{\text{sing}}$ for every $b \in B$;
        \item Fiber by fiber, $\tau$ is a generically fixed point free involution.
    \end{itemize}
    To address the first bullet point, consider $\sS' \to \sD$ the inclusion of the $g$-singular locus. This is a closed embedding, and consider the following fibred diagram:
    $$\xymatrix{F \ar[d]_h \ar[r] & \sS' \ar[d] \\ \bigsqcup_{i=1}^n \sigma_i(\sH_3) \ar[r] & \sD}$$
    Then $h$ will be a closed embedding. Requiring $h$ to be an isomorphism is equivalent to the first bullet point.
    Then the flattening stratification guarantees that there is a well-defined closed substack where $h$ is an isomorphism. In other terms, up to replacing $\sH_3$ with a locallly closed substack, we can assume that the first bullet point is satisfied.

Finally, being an isomorphism is an open condition, so there is an open substack of $\cH_3$ where $\tau$ is an isomorphism. Observe now that if $\sigma$ is an isomorphism of an orbifold nodal curve, such that it agrees with the identity on an open dense subset, then $\sigma=\Id$. Thus if $\tau$ fixes a generic point of $\cC$, then it fixes the irreducible component that is its closure.

Consider then the following fiber diagram:
$$ \xymatrix{F \ar[d]_\psi \ar[r] & \sC \ar[d]^{\text{Diag}} \\
\sC \ar[r]_-{(\Id,\tau)} & \sC \times_{\cH_3}\sC.}$$
$F$ is the fixed locus of $\tau$. We need to cut out the locus where $F$ contains no irreducible components of $\cC$. This is equivalent to $F \to \cH_3$ being finite so by semi-continuity of fiber dimension, there is an open subset where $\tau$ fixes no generic points. Similarly, to ensure that $\tau^2=\Id$, we can replace in the diagram above $\tau$ with $\tau^2$. Let $F'$ be the new fiber product we obtain. Then the locus where $\tau^2=\Id$ is the locus where $\pi_1:F' \to \sC$ is surjective. Or in other therms, where the kernel of the map $\cO_{\sC} \to (\pi_1)_*\cO_{F'}$ is zero. This is the locus where the support of $\operatorname{Coker}(\cO_{\sC} \to (\pi_1)_*\cO_{F'})$ is empty: it is an open substack of $\cH_3$.
\end{proof}

Putting this together with Theorem \ref{Teo:gluing:morphisms} we obtain:

\begin{theorem}\label{Teo:gluing:morphism:between:moduli} There is a functorial gluing morphism
$\cG_{2,v} \to \cK_{2,v}^\omega$ from the stack of gluing data to the stack of twisted stable surfaces. \end{theorem} 

\subsection{The boundary strata of $\cK_{2,v}^\omega$}\label{subsec:boundary:strata}

In this subsection, we show that there is a locally closed stratification of $\cK_{2,v}^\omega$
of equinormalizable surfaces with equinormalizable double locus which are the images of the gluing morphisms above. 

\begin{Lemma}\label{Lemma:sim:norm:curve} Let $f : \cY \to B$ be a proper family of generically reduced DM stacks over a base scheme $B$.
Then there exists a locally closed stratification of $B$ over which $f$ is simultaneously normalizable.
\end{Lemma}
\begin{proof} Using Noetherian induction, it suffices to prove that there is a nonempty open subset of $B$ where $f$ is simultaneously normalizable.
Up to replacing $B$ with its reduced structure, we can assume that $B$ is reduced. Then it is generically smooth, so up to further shrinking $B$ we can assume that it is smooth and connected.
Consider $\nu:\cY^n \to \cY$ the normalization.
Up to shrinking $B$ we can assume that $\cY^n \to B$ is flat. The generic geometric fiber is normal and the locus $\cU$ in $\cY$ where the fibers are
normal is open from
\cite{EGAIV}*{Théorème 12.2.6}. So its complement $\cZ:=\cY \smallsetminus \cU$ is closed, and since $f$ is proper,
$f(\cZ)$ is closed too. Then up to shrinking $B$ we can assume that $\cY^n \to B$ has normal fibers. 

But $\nu$ is an isomorphism on the smooth locus of the morphism $\cY \to B$, since $B$ is normal. In particular, for every $b \in B$,
the map $\cY_b^n \to \cY_b$ is an isomorphism at the generic points of $\cY_b$, and it is finite. So it is a simultaneous normalization.
 \end{proof}
\begin{Prop}\label{Prop:the:stratification:is:finite}
There is a finite, locally closed stratification of $\cK_{2,v}^\omega$ such that
each stratum is the image of a family of gluing data under the gluing morphism $\cG_{2,v} \to \cK_{2,v}^\omega$.
In particular, taking the scheme theoretic images of components under the gluing morphism stratifies $\cK_{2,v}^\omega$
into a finite, locally closed union of boundary components. 
\end{Prop}
\begin{proof} Let $\sX\to \cK_{n,v}^\omega$ be the universal twisted stable surface.
From Lemma \ref{Lemma:sim:norm:curve} there is a locally closed embedding $\sS \to \cK_{2,v}^\omega$
where the fibers of $\sX\to \cK_{n,v}^\omega$ admit a simultaneous normalization. Let $\sX':=\sX \times_{\cK_{2,v}^\omega}\sS $ be the pull back, and let $\sD$ be the closed substack which cuts the fiberwise double locus.
Again from Lemma \ref{Lemma:sim:norm:curve}, there is a locally closed embedding where
the fibers of $\sD \to \sS$ are
simultaneously normalizable. So up to further stratifying $\sS$ we can assume that both $\sX'$ and $\sD$ admit a simultaneous normalization. Let $\nu:\sD^n \to \sD$ be the simultaneous normalization of $\sD \to \sS$. 

Let $\sZ$ be the singular locus of $\sD \to \sS$, and $F$ its pull back through $\nu$.
Up to stratifying further $\sS$, we can assume that $F \to \sS$ is étale, and up to replacing $\sS$ with an étale cover, we can assume that there are disjoint sections sections $\sigma_i:\sS \to \sD^n$ which surject to $F$. 

The composition $\sD^n \to \sX'$ is generically a $2$ to $1$ cover onto its image so $\sD^n$ is equipped with a rational involution $\tau$.
After further stratifying and applying Lemma \ref{Lemma:we:have:tau:on:the:stack} we can assume that $\tau$ is a morphism. One can check fiber by
fiber to see that $\tau$ preserves $\sqcup \mathrm{Im}(\sigma_i)$. This gives a family of gluing data whose image in $\cK_{2,v}^\omega$
is this stratum. Finally, by \cite{Kollarsingmmp}*{Theorem 5.13}, every point of $\cK_{2,v}^\omega$ corresponding to a non-normal twisted stable surface lies in such a stratum. 
\end{proof}

\begin{Remark} 
Proposition \ref{Prop:the:stratification:is:finite} only posits the existence of some finite stratification and says nothing about how to enumerate the strata, nor the
components of $\cG_{2,v}$.
We include it to rule out pathological behaviour like the image of $\cG_{2,v}$ being an infinite disjoint union of points.
One hopes for a more functorial stratification described in terms of combinatorial and numerical
data of the surfaces, as well as a generalization of functorial gluing morphisms to higher dimensions.
Doing this will require a generalization of Koll\'ar's theory of hulls and husks, e.g. \cite{Kol11},
to cyclotomic stacks, which we will pursue in a future article. 
\end{Remark}

\section{The deformation space of an elliptic K3 surface with a section and a fiber}\label{section_example}

The goal of this section is to give an explicit computation of the deformations and obstructions for Koll\'ar families of stable pairs using the formalism of twisted stable pairs. 

Consider a pair $(Y,E+F)$ consisting of a generic elliptic K3 surface with section $E$ and smooth fiber $F$. This pair is of log general type and its log canonical model $(X,D)$ is obtained by contracting the section $p : Y \to X$ and taking $D = p_*E$. Since $p$ contracts a $(-2)$-curve, $(X,D)$ is a Gorenstein stable pair with an $A_1$-singularity and $D$ is a smooth curve passing through the singular point. Note in particular that $D$ is not Cartier but $2D$ is. 

Since $E$ is rigid and $F$ moves in a pencil, the pair $(Y,E+F)$ is parametrized by a smooth 19-dimensional moduli space sitting inside a $\mathbb{P}^1$-bundle over the 18-dimensional moduli space of $U$-polarized K3 surfaces where $U$ is the standard hyperbolic plane (see e.g. \cite{k3paper}). Thus one might expect that the moduli space of the stable pair $(X,D)$ should agree with the $19$-dimensional moduli space of $(Y, E+F)$ described above. It is not hard to see that this is the case set theoretically but \emph{a priori} the infinitessimal deformation space of $(X,D)$ might be larger (see \ref{remark:def}).

Analytically locally around the $A_1$-singularity, the pair $(X,D)$ is isomorphic to the quotient of $(\mathbb{A}^2, \{x = 0\})$ by the action $(x,y) \mapsto (-x,-y)$ so the associated AH stack $\cX$ is smooth with Cartier boundary divisor $\cD$. In particular, the data of the divisor $\cD$ is equivalent to the data of a morphism $q: \cX \to \Theta:= [\mathbb{A}^1/\mathbb{G}_m]$, and to understand the deformations and obstructions of $(\cX,\cD)$ it suffices to compute the deformations and obstructions of $q$. Note also that $X$ is Gorenstein, $p$ is crepant, and the action fixes $dx \wedge dy$ so $\omega_X \cong \cO_X$ and $\omega_\cX \cong \cO_\cX$ are both trivial. 

A standard computation shows that the cotangent complex $\bL_{\Theta} \simeq_{qis} [\cO_\Theta \xrightarrow{\cdot x} \cO_\Theta]$ in degrees $[0,1]$ where $x$ is the $\mathbb{G}_m$-invariant coordinate on $\mathbb{A}^1$. Moreover, the map $q: \cX \to \Theta$ induced by the Cartier divisor $\cD$ is flat so $q^*\bL_\Theta[1] \simeq_{qis} \cO_\cD$. Since $\cX$ is smooth over $k$, the transitivity triangle for the composition $\cX \to \Theta \to \spec(k)$ shifted by $1$ yields
$$
\Omega^1_\cX[0] \to \bL_{\cX/\Theta} \to \cO_\cD[0] \xrightarrow{+1}.
$$
In particular, $\cH^i(\bL_{\cX/\Theta}) = 0$ if $i \neq 0$ and we have a short exact sequence
$$
0 \to \Omega^1_{\cX} \to \cH^0(\bL_{\cX/\Theta}) \to \cO_\cD \to 0 
$$
which identifies $\cH^0(\bL_{\cX/\Theta})$ with the sheaf of logarithmic differentials $\Omega^1_{\cX}(\log \cD)$.

We will see (Propositions \ref{prop:def:X} \& \ref{prop:def:tsp}) that $Def(\cX)$ and $Def(\cX, \cD)$ are both smooth and $19$-dimensional. One can interpret this as follows. $Def(\cX)$ is identified with the deformation space of $(Y,E)$, and the $18$-dimensional $U$-polarized deformation space of $Y$ is identified with the deformations of the polarized pair $(\cX, \cL)$. Finally the $19$-dimensional deformation space of $(\cX, \cD)$ over $Def(\cX, \cD)$ can be identified with the $\mathbb{P}^1$-bundle described above.  

\begin{Remark}\label{remark:def} One can compute that the deformation space of the coarse space $X$ is smooth and $20$-dimensional. Moreover there are both local and global obstructions to deforming the non-Cartier subscheme $D$. The obstructions live in a $2$-dimensional space $\E^1(\Omega_X, \cO_D)$ and the obstruction map is in fact surjective. Thus the naive deformation space of $(X,D)$ is also $19$-dimensional, matching the deformation space of the twisted stable pair $(\cX,\cD)$. We thus conclude that all deformations of the pair $(X,D)$ lift to deformations of $(\cX,\cD)$ and in particular they all satisfy Koll\'ar's condition. However, this is far from clear \emph{a priori} and we only know how to check it by first computing the Koll\'ar deformations using the twisted stable pair $(\cX, \cD)$. 
\end{Remark}

\subsection{Deformation theory of $\cX$}\label{subsection_def_obs_for_XX}
The goal of this subsection is to compute $\E^i(\Omega^1_{\cX},\cO_{\cX})\cong H^i(\cX,T_{\cX})$, where $T_{\cX}$ denotes the tangent bundle of $\cX$. 

Let $f:\cX\to X$ be the coarse space map. Since $f$ is finite, surjective and cohomologically affine, 
$H^i(\cX,T_{\cX}) \cong H^i(X,f_*T_{\cX})$ and the sheaf $f_*T_{\cX}$ is reflexive. Consider now the contraction of the section $p:Y\to X$ where $Y$ is the elliptic K3 surface with section $E$ as above. The holomorphic symplectic form on $Y$ gives an identification $T_Y \cong \Omega^1_Y$ so by \cite[Theorem 1.4]{GKKP} we have that $p_*T_Y$ is reflexive. Both these sheaves are also isomorphic to the tangent sheaf $T_X$ over a big open subset of $X$ so by reflexivity we have
$$
p_*T_Y \cong T_X \cong f_*T_\cX. 
$$

By the Leray spectral sequence we have $H^0(X,T_X) \cong H^0(Y,T_Y)$ and 
$$0 \to H^1(X,T_X) \to H^1(Y,T_Y) \to H^0(X,R^1p_*T_Y) \to H^2(X,T_X) \to H^2(Y,T_Y).$$
Hodge theory of K3 surfaces gives us that
$$
h^{i}(Y, T_Y) =\left\{ \begin{array}{lr} 0 & i = 0 \\ 20 & i = 1 \\ 0 & i = 2\end{array} \right.
$$

\begin{Lemma}\label{lemma_R1push_forward_tangent_is_skyscraper}
With the notations above, we have $R^1p_*T_Y = H^1(E,T_Y|_E)_x \cong k_x$ is a skyscraper sheaf supported at the singular point of $X$.
\end{Lemma}\begin{proof} Since $p$ induces an isomorphism $Y \setminus E \to X \setminus p$, $R^1p_*T_Y$ is supported at $x$. Then
$$
R^1p_*T_Y = (R^1p_*T_Y)_x^{\widehat{}} = \varprojlim_{n} H^1(E_n, T_Y|_{E_n})
$$
by the theorem on formal functions where $E_n$ is the $n^{th}$ infinitessimal thickening of $E$ defined by ideal sheaf $\cI^n$. .

Since $E\cong \mathbb{P}^1$ and $E^2 = -2$, the normal bundle sequence for the regular embedding $i : E \to Y$ is 
$$
0 \to \cO_{\mathbb{P}^1}(2) \to T_Y|_{E} \to \cO_{\mathbb{P}^1}(-2) \to 0
$$
which splits since $\mathrm{Ext}^1(\cO_{\mathbb{P}^1}(-2),\cO_{\mathbb{P}^1}(2))=0$, giving $T_Y|_{E} \cong \cO_E(2) \oplus \cO_E(-2)$ and $H^1(E, T_Y|_E) = 1$.  

Consider the closed embedding $E_n \hookrightarrow E_{n + 1}$ for $n \geq 1$. There is a short exact sequence $0 \to \cI^n/\cI^{n + 1} \to \cO_{E_{n + 1}}  \to \cO_{E_n}  \to 0$ with $\cI^n/\cI^{n + 1} \cong \cO_E(2n)$, the $n^{th}$ power of the conormal bundle $\cO_E(2)$ of $E \subset Y$. Tensoring with $T_Y|_E$ we obtain
$$
0 \to \cO_E(2n + 2) \oplus \cO_E(2n - 2) \to T_Y|_{E_{n + 1}} \to T_Y|_{E_n} \to 0. 
$$
Since $H^1(\mathbb{P}^1, \cO(2n + 2) \oplus O(2n - 2)) = 0$ for all $n \geq 1$, we conclude that the map $H^1(E_{n + 1}, T_Y|_{E_{n + 1}}) \to H^1(E_n, T_Y|_{E_n})$ is an isomorphism and so 
$$
\varprojlim_n H^1(E_n, T_Y|_{E_n}) = H^1(E, T_Y|_E) = k.
$$
\end{proof}

Under the canonical isomorphism in Lemma \ref{lemma_R1push_forward_tangent_is_skyscraper}, the edge map $H^1(Y, T_Y) \to H^0(X, R^1p_*T_Y)$ can be identified with the restriction $H^1(Y, T_Y) \to H^1(E, T_Y|_E)$. 

\begin{Lemma}\label{lemma_relative_differentials_for_g} Let $g : Y \to \mathbb{P}^1$ be an elliptically fibered $K3$ surface with section $E$. Then
\begin{enumerate}
\item $H^1(Y,T_Y) \to H^1(Y,(T_Y|_E)$ is surjective, and 
\item $H^0(\Omega^1_g\otimes \cO_Y(E))=0$.
\end{enumerate}
\end{Lemma}
\begin{proof}
By taking the long exact sequence in cohomology for $0\to T_Y(-E) \to T_Y \to T_Y|_E \to 0$, for part $(i)$ we need to check that $H^2(T_Y(-E)) = 0$. Using Serre duality, this is equivalent to $$H^0(\Omega^1_Y(E)) = 0.$$
We have a short exact sequence of differentials
\begin{equation}\label{eqn:omega:1}
0 \to g^*\Omega^1_{\mathbb{P}^1} \to \Omega^1_Y \to \Omega^1_g \to 0.
\end{equation}
Moreover, $\omega_g$ is a line a bundle which agrees with $\Omega^1_g$ on the big open set $U \subset Y$ where $g$ is smooth. Since the singular fibers of $g$ have nodal singularities, $\Omega^1_g$ is torsion free and there is a canonical map $\Omega^1_g \to \omega_g$ identifying $\omega_g$ with the reflexive hull. The cokernel $F$ is a torsion sheaf supported at the nodes of the singular fibers. A local computation around each such node $p \in Y$ shows that $F_p = k$ is rank $1$. Thus we have a short exact sequence 
\begin{equation}\label{equation:omega:2}
0 \to \Omega^1_g \to \omega_g \to F\to 0 
\end{equation}
where $F$ is a direct sum of $k_p$ over the 24 nodes of the 24 singular fibers. 

Now, $\cO_Y \cong \omega_Y \cong \omega_g \otimes g^*(\omega_{\mathbb{P}^1})$ so $\omega_g\cong  g^*(\cO_{\mathbb{P}^1}(2))$. Then $H^0(\omega_g(E)) = H^0(\cO_{\mathbb{P}^1}(2)\otimes g_*\cO(E))$ by projection formula. By \cite[II.3.7]{Miranda} $g_*\cO(E) \cong \cO_{\mathbb{P}^1}$ so $H^0(\cO_{\mathbb{P}^1}(2)\otimes g_*\cO(E))\cong k^{\oplus 3}$. Then twisting the exact sequence (\ref{eqn:omega:1}) with $\cO(E)$ and computing cohomology we have
$$0 \to H^0(\Omega^1_g(E))\to H^0(\omega_g(E)) \to H^0(F(E)).$$
To check that the morphism $H^0(\omega_g(E)) \to H^0(F(E))$ is injective, it suffices to check that for every section $s \in H^0(\omega_g(E))$ there is a nodal point $p$ of a singular fiber where $s$ does not vanish. We identified sections of $\omega_g(E)$ above with degree $2$ polynomials on $\mathbb{P}^1$ which vanish along at most $2$ fibers, proving the claim as there are $24$ singular fibers. Therefore $H^0(\Omega^1_g(E)) = 0$, proving $(ii)$. 

Returning to $(i)$, we twist (\ref{equation:omega:2}) by $E$ and note that $\Omega^1_{\mathbb{P}^1} = \cO_{\mathbb{P}^1}(-2)$ so $H^0(g^*\Omega^1_{\mathbb{P}^1}(E)) = H^0(\cO(-2)) = 0$ using projection formula and \cite[II.3.7]{Miranda} as before. Then  $H^0(\Omega^1_Y(E)) \to H^0(\Omega^1_g(E))$ is injective so $H^0(\Omega^1_Y(E)) = 0$ as required.  
\end{proof}

Putting this all together, we conclude the following:

\begin{Prop}\label{prop:def:X}The deformation space $Def(\cX)$ is smooth and 19-dimensional. \end{Prop}

\begin{proof}
Using the long exact sequence
$$0 \to H^1(X,T_X) \to H^1(Y,T_Y) \to H^0(X,R^1p_*T_Y) \to H^2(X,T_X) \to H^2(Y,T_Y) = 0$$
and surjectivity of $H^1(Y,T_Y) \to H^0(X, R^1p_*T_Y)$ we have $H^1(X,T_X)\cong k^{\oplus 19}$ and $H^2(X,T_X) = 0$ so the deformation space is unobstructed.
\end{proof}

\subsection{Deformations theory of the twisted stable pair $(\cX,\cD)$.}
Following the discussion at the start of this section, we wish to compute $\E^i(\Omega^1_{\cX}(\log\cD),\cO_{\cX}) = H^i(X, T_\cX(-\log \cD))$. Consider the residue exact sequence 
\begin{equation}\label{eqn:log:sequence}
0 \to \Omega^1_\cX \to \Omega^1_\cX(\log \cD) \to \cO_\cD \to 0.
\end{equation}
\begin{Lemma}
$\E^2(\Omega^1_\cX(\log \cD), \cO_\cX) \cong H^0(\cX, \Omega^1_\cX(\log \cD)) = 0$.
\end{Lemma}
\begin{proof} The first equality is Serre duality. Consider the connecting homomorphism $k = H^0(\cX, \cO_\cD) \to H^1(\cX, \Omega^1_\cX)$. One can check that this map sends $1$ to $c_1(\cL)$ where $\cL = \cO_\cX(\cD)$. Since $\cD$ is ample and $\cX$ is proper, $c_1(\cL)\neq 0$ so the connecting homomorphism is injective. Therefore $H^0(\cX, \Omega^1_\cX) \to H^0(\cX, \Omega^1_\cX(\log \cD))$ is an isomorphism but $H^0(\cX, \Omega^1_\cX) = H^0(X, f_*\Omega^1_\cX) = H^0(X, p_*\Omega^1_Y) = H^0(Y, \Omega^1_Y) = 0$. 
\end{proof}

\begin{Lemma}\label{lemma_ext1_OD}
$$
\E^i(\cO_D,\cO_{\cX}) \cong \left\{\begin{array}{lr} 0 & i \neq 1,2 \\ k & i = 1,2\end{array}\right.
$$ 
\end{Lemma}
\begin{proof}
We apply the functor $RHom(-,\cO_{\cX})$ to the sequence $0\to \cO_{\cX}(-\cD) \to \cO_{\cX} \to \cO_D \to 0$ using the fact that $\E^i(\cL,\cO_{\cX}) = H^i(\cX,\cL^{-1})$ for line bundles $\cL$. Moreover, since $f:\cX \to X$ is cohomologically affine and $X$ has rational singularities, $R^if_*\cO_{\cX} = R^ip_*\cO_{Y} = 0$ for $i>0$ so $H^i(\cX,\cO_\cX) = H^i(Y,\cO_Y)$. We conclude that $H^0(\cX,\cO_\cX)=H^2(\cX,\cO_\cX)=k$, and $H^1(\cX,\cO_\cX)=0$. Putting this together, we have the following short exact sequences. 
\begin{equation}\label{eqn:les:1}
0 \to k \to H^0(\cX,\cO_{\cX}(\cD))\to \E^1(\cO_{\cD},\cO)
\to 0
\end{equation}

\begin{equation}\label{eqn:les:2}
0 \to H^1(\cX,\cO_{\cX}(\cD)) \to \E^2(\cO_{\cD},\cO)\to k \to 0
\end{equation}
By Serre duality, $H^0(\cX, \cO_\cX(\cD)) = H^2(\cX, \cO_\cX(-\cD))$ and $H^1(\cX, \cO_\cX(\cD)) = H^1(\cX, \cO_\cX(-\cD))$. Taking the long exact sequence of cohomology of $0\to \cO_{\cX}(-\cD) \to \cO_{\cX} \to \cO_{\cD} \to 0$ we have that $H^1(\cX,\cO_{\cX}(-\cD)) = 0$ and $H^2(\cX,\cO_{\cX}(-\cD)) = k^{\oplus 2}$. Here we have used that $H^1(\cX, \cO_\cX) = 0$ and $H^1(\cD, \cO_\cD) = 1$ since $\cD$ is genus $1$. Applying this to (\ref{eqn:les:1}) and $(\ref{eqn:les:2})$ finishes the proof. 
\end{proof}

Putting this all together, we conclude the following.

\begin{Prop}\label{prop:def:tsp}
The deformation space of the twisted stable pair $Def(\cX, \cD)$ is smooth and $19$-dimensional. 
\end{Prop} 
\begin{proof}
Note that $\Hom(\Omega^1_\cX(\log \cD), \cO_\cX) = H^0(\cX,T_\cX(-\log \cD))$ vanishes as it injects into $H^0(\cX, T_\cX) = 0$. Using this and applying $RHom(-,\cO_\cX)$ to the short exact sequence (\ref{eqn:log:sequence}), we obtain 
\begin{align*}
0 \to \E^1(\cO_\cD, \cO_\cX) \to H^1(\cX, T_\cX(-\log \cD)) &\to H^1(\cX,T_\cX) \to \\ \E^1(\cO_\cD, \cO_\cX) &\to H^2(\cX, T_\cX(-\log \cD)) \to 0.
\end{align*}
Applying the previous results, we obtain
$$
0 \to k \to H^1(\cX, T_\cX(-\log \cD)) \to k^{\oplus 19} \to k \to 0
$$
and $H^2(\cX, T_\cX(-\log \cD) = 0$ so we conclude that $H^1(\cX, T_\cX(-\log \cD)$ is $19$-dimensional and $Def(\cX,\cD) = Def(\cX \to \Theta)$ is unobstructed as claimed.
\end{proof}

\bibliographystyle{amsalpha}
\bibliography{tsp}

\begin{bibdiv}
\begin{biblist}

\bib{k3paper}{misc}{
      author={Ascher, Kenneth},
      author={Bejleri, Dori},
       title={Compact moduli of elliptic k3 surfaces},
        date={2019},
}

\bib{AH}{incollection}{
      author={Abramovich, Dan},
      author={Hassett, Brendan},
       title={Stable varieties with a twist},
        date={2011},
   booktitle={Classification of algebraic varieties},
      series={EMS Ser. Congr. Rep.},
   publisher={Eur. Math. Soc., Z\"{u}rich},
       pages={1\ndash 38},
         url={https://doi.org/10.4171/007-1/1},
      review={\MR{2779465}},
}

\bib{AK16}{article}{
      author={Altmann, Klaus},
      author={Koll{\'a}r, J{\'a}nos},
       title={The dualizing sheaf on first-order deformations of toric surface
  singularities},
        date={2016},
     journal={Journal f{\"u}r die reine und angewandte Mathematik (Crelles
  Journal)},
}

\bib{al1}{article}{
      author={Alexeev, Valery},
       title={Boundedness and {$K^2$} for log surfaces},
        date={1994},
        ISSN={0129-167X},
     journal={Internat. J. Math.},
      volume={5},
      number={6},
       pages={779\ndash 810},
         url={http://dx.doi.org/10.1142/S0129167X94000395},
      review={\MR{1298994}},
}

\bib{AOV}{article}{
      author={Abramovich, Dan},
      author={Olsson, Martin},
      author={Vistoli, Angelo},
       title={Twisted stable maps to tame {A}rtin stacks},
        date={2011},
        ISSN={1056-3911},
     journal={J. Algebraic Geom.},
      volume={20},
      number={3},
       pages={399\ndash 477},
         url={https://doi.org/10.1090/S1056-3911-2010-00569-3},
      review={\MR{2786662}},
}

\bib{BCHM}{article}{
      author={Birkar, Caucher},
      author={Cascini, Paolo},
      author={Hacon, Christopher~D.},
      author={McKernan, James},
       title={Existence of minimal models for varieties of log general type},
        date={2010},
        ISSN={0894-0347},
     journal={J. Amer. Math. Soc.},
      volume={23},
      number={2},
       pages={405\ndash 468},
         url={https://doi-org.libproxy.mit.edu/10.1090/S0894-0347-09-00649-3},
      review={\MR{2601039}},
}

\bib{BH93}{book}{
      author={Bruns, Winfried},
      author={Herzog, J\"{u}rgen},
       title={Cohen-{M}acaulay rings},
      series={Cambridge Studies in Advanced Mathematics},
   publisher={Cambridge University Press, Cambridge},
        date={1993},
      volume={39},
        ISBN={0-521-41068-1},
}

\bib{Conrad}{book}{
      author={Conrad, Brian},
       title={Grothendieck duality and base change},
      series={Lecture Notes in Mathematics},
   publisher={Springer-Verlag, Berlin},
        date={2000},
      volume={1750},
        ISBN={3-540-41134-8},
}

\bib{DH18}{article}{
      author={Deopurkar, Anand},
      author={Han, Changho},
       title={Stable log surfaces, admissible covers, and canonical curves of
  genus 4},
        date={2018},
     journal={arXiv preprint arXiv:1807.08413},
}

\bib{DM69}{article}{
      author={Deligne, P.},
      author={Mumford, D.},
       title={The irreducibility of the space of curves of given genus},
        date={1969},
        ISSN={0073-8301},
     journal={Inst. Hautes \'{E}tudes Sci. Publ. Math.},
      number={36},
       pages={75\ndash 109},
}

\bib{GKKP}{article}{
      author={Greb, Daniel},
      author={Kebekus, Stefan},
      author={Kov\'{a}cs, S\'{a}ndor~J.},
      author={Peternell, Thomas},
       title={Differential forms on log canonical spaces},
        date={2011},
        ISSN={0073-8301},
     journal={Publ. Math. Inst. Hautes \'{E}tudes Sci.},
      number={114},
       pages={87\ndash 169},
         url={https://doi.org/10.1007/s10240-011-0036-0},
      review={\MR{2854859}},
}

\bib{EGAIV}{article}{
      author={Grothendieck, Alexander},
       title={{\'E}l{\'e}ments de g{\'e}om{\'e}trie alg{\'e}brique
  (r{\'e}dig{\'e}s avec la collaboration de {J}ean {D}ieudonn{\'e}): {IV}.
  {\'e}tude locale des sch{\'e}mas et des morphismes de sch{\'e}mas,
  troisi{\'e}me partie},
        date={1966},
     journal={Inst. Hautes {\'E}tudes Sci. Publ. Math},
      volume={28},
}

\bib{Hac04}{article}{
      author={Hacking, Paul},
       title={Compact moduli of plane curves},
        date={2004},
        ISSN={0012-7094},
     journal={Duke Math. J.},
      volume={124},
      number={2},
       pages={213\ndash 257},
}

\bib{HK}{article}{
      author={Hassett, Brendan},
      author={Kov\'acs, S\'andor},
       title={Reflexive pull-backs and base extension},
        date={2004},
        ISSN={1056-3911},
     journal={J. Algebraic Geom.},
      volume={13},
      number={2},
       pages={233\ndash 247},
}

\bib{HMX}{article}{
      author={Hacon, Christopher~D.},
      author={McKernan, James},
      author={Xu, Chenyang},
       title={Boundedness of moduli of varieties of general type},
        date={2018},
        ISSN={1435-9855},
     journal={J. Eur. Math. Soc. (JEMS)},
      volume={20},
      number={4},
       pages={865\ndash 901},
         url={https://doi-org.libproxy.mit.edu/10.4171/JEMS/778},
      review={\MR{3779687}},
}

\bib{HX}{article}{
      author={Hacon, Christopher~D.},
      author={Xu, Chenyang},
       title={Existence of log canonical closures},
        date={2013},
        ISSN={0020-9910},
     journal={Invent. Math.},
      volume={192},
      number={1},
       pages={161\ndash 195},
         url={https://doi-org.libproxy.mit.edu/10.1007/s00222-012-0409-0},
      review={\MR{3032329}},
}

\bib{karu}{article}{
      author={Karu, Kalle},
       title={Minimal models and boundedness of stable varieties},
        date={2000},
        ISSN={1056-3911},
     journal={J. Algebraic Geom.},
      volume={9},
      number={1},
       pages={93\ndash 109},
      review={\MR{1713521}},
}

\bib{KK18}{article}{
      author={Koll\'ar, J\'anos},
      author={Kov\'acs, S\'andor},
       title={Deformations of log canonical and f-pure singularities},
        date={2018},
     journal={ArXiv e-prints},
      eprint={1807.07417},
}

\bib{KM98}{book}{
      author={Koll\'{a}r, J\'{a}nos},
      author={Mori, Shigefumi},
       title={Birational geometry of algebraic varieties},
      series={Cambridge Tracts in Mathematics},
   publisher={Cambridge University Press, Cambridge},
        date={1998},
      volume={134},
        ISBN={0-521-63277-3},
        note={With the collaboration of C. H. Clemens and A. Corti, Translated
  from the 1998 Japanese original},
}

\bib{Knu}{article}{
      author={Knudsen, Finn~F.},
       title={The projectivity of the moduli space of stable curves. {III}.
  {T}he line bundles on {$M_{g,n}$}, and a proof of the projectivity of
  {$\overline M_{g,n}$} in characteristic {$0$}},
        date={1983},
        ISSN={0025-5521},
     journal={Math. Scand.},
      volume={52},
      number={2},
       pages={200\ndash 212},
}

\bib{Kol08}{article}{
      author={Koll\'{a}r, J\'{a}nos},
       title={Hulls and husks},
        date={2008},
     journal={arXiv preprint arXiv:0805.0576},
}

\bib{Kol11}{article}{
      author={Koll\'{a}r, J\'{a}nos},
       title={Simultaneous normalization and algebra husks},
        date={2011},
        ISSN={1093-6106},
     journal={Asian J. Math.},
      volume={15},
      number={3},
       pages={437\ndash 449},
}

\bib{Kol13}{incollection}{
      author={Koll\'{a}r, J\'{a}nos},
       title={Moduli of varieties of general type},
        date={2013},
   booktitle={Handbook of moduli. {V}ol. {II}},
      series={Adv. Lect. Math. (ALM)},
      volume={25},
   publisher={Int. Press, Somerville, MA},
       pages={131\ndash 157},
}

\bib{Kollarsingmmp}{book}{
      author={Koll\'{a}r, J\'{a}nos},
       title={Singularities of the minimal model program},
      series={Cambridge Tracts in Mathematics},
   publisher={Cambridge University Press, Cambridge},
        date={2013},
      volume={200},
        ISBN={978-1-107-03534-8},
        note={With a collaboration of S\'{a}ndor Kov\'{a}cs},
}

\bib{kol1}{book}{
      author={Koll\'{a}r, J\'{a}nos},
       title={Families of varieties of general type},
   publisher={Available at \url{https://web.math.princeton.edu/~kollar/}},
        date={2017},
}

\bib{Kol_div19}{article}{
      author={Koll\'{a}r, J\'{a}nos},
       title={Families of divisors},
        date={2019},
     journal={arXiv preprint arXiv:1910.00937},
}

\bib{kol}{article}{
      author={Koll\'{a}r, J\'{a}nos},
       title={Projectivity of complete moduli},
        date={1990},
        ISSN={0022-040X},
     journal={J. Differential Geom.},
      volume={32},
      number={1},
       pages={235\ndash 268},
         url={http://projecteuclid.org/euclid.jdg/1214445046},
      review={\MR{1064874}},
}

\bib{Kov18}{article}{
      author={{Kov{\'a}cs}, S.},
       title={Rational singularities},
        date={2017},
     journal={ArXiv e-prints},
      eprint={1703.02269},
}

\bib{KP}{article}{
      author={Kov\'acs, S\'andor},
      author={Patakfalvi, Zsolt},
       title={Projectivity of the moduli space of stable log-varieties and
  subadditivity of log-{K}odaira dimension},
        date={2017},
        ISSN={0894-0347},
     journal={J. Amer. Math. Soc.},
      volume={30},
      number={4},
       pages={959\ndash 1021},
}

\bib{KSB}{article}{
      author={Koll\'{a}r, J.},
      author={Shepherd-Barron, N.~I.},
       title={Threefolds and deformations of surface singularities},
        date={1988},
        ISSN={0020-9910},
     journal={Invent. Math.},
      volume={91},
      number={2},
       pages={299\ndash 338},
}

\bib{hom}{article}{
      author={Lieblich, Max},
       title={Remarks on the stack of coherent algebras},
        date={2006},
        ISSN={1073-7928},
     journal={Int. Math. Res. Not.},
       pages={Art. ID 75273, 12},
}

\bib{YN18}{article}{
      author={Lee, Yongnam},
      author={Nakayama, Noboru},
       title={Grothendieck duality and {$\mathbb{Q}$}-{G}orenstein morphisms},
        date={2018},
        ISSN={0034-5318},
     journal={Publ. Res. Inst. Math. Sci.},
      volume={54},
      number={3},
       pages={517\ndash 648},
         url={https://doi.org/10.4171/PRIMS/54-3-3},
      review={\MR{3834279}},
}

\bib{Mat89}{book}{
      author={Matsumura, Hideyuki},
       title={Commutative ring theory},
     edition={Second},
      series={Cambridge Studies in Advanced Mathematics},
   publisher={Cambridge University Press, Cambridge},
        date={1989},
      volume={8},
        ISBN={0-521-36764-6},
        note={Translated from the Japanese by M. Reid},
      review={\MR{1011461}},
}

\bib{Miranda}{book}{
      author={Miranda, Rick},
       title={The basic theory of elliptic surfaces},
      series={Dottorato di Ricerca in Matematica. [Doctorate in Mathematical
  Research]},
   publisher={ETS Editrice, Pisa},
        date={1989},
      review={\MR{1078016}},
}

\bib{Ols06}{article}{
      author={Olsson, Martin~C},
       title={Deformation theory of representable morphisms of algebraic
  stacks},
        date={2006},
     journal={Mathematische Zeitschrift},
      volume={253},
      number={1},
       pages={25\ndash 62},
}

\bib{Ols16}{book}{
      author={Olsson, Martin},
       title={Algebraic spaces and stacks},
      series={American Mathematical Society Colloquium Publications},
   publisher={American Mathematical Society, Providence, RI},
        date={2016},
      volume={62},
        ISBN={978-1-4704-2798-6},
}

\bib{Pat16}{article}{
      author={Patakfalvi, Zsolt},
       title={Fibered stable varieties},
        date={2016},
        ISSN={0002-9947},
     journal={Trans. Amer. Math. Soc.},
      volume={368},
      number={3},
       pages={1837\ndash 1869},
}

\bib{Rom05}{article}{
      author={Romagny, Matthieu},
       title={Group actions on stacks and applications},
        date={2005},
        ISSN={0026-2285},
     journal={Michigan Math. J.},
      volume={53},
      number={1},
       pages={209\ndash 236},
}

\bib{Rydh}{article}{
      author={Rydh, David},
       title={Compactification of tame deligne--mumford stacks},
        date={2014},
     journal={Preprint},
}

\bib{ks}{misc}{
      author={Schwede, Karl},
       title={Obtaining non-normal varieties by pushout},
         how={MathOverflow},
         url={https://mathoverflow.net/q/186650},
        note={URL:https://mathoverflow.net/q/186650},
}

\bib{Wlo05}{article}{
      author={W\l~odarczyk, Jaros\l~aw},
       title={Simple {H}ironaka resolution in characteristic zero},
        date={2005},
        ISSN={0894-0347},
     journal={J. Amer. Math. Soc.},
      volume={18},
      number={4},
       pages={779\ndash 822},
}

\end{biblist}
\end{bibdiv}

\end{document}